\documentclass[12pt,a4paper]{article}
\usepackage{oldgerm}
\usepackage{amssymb}
\usepackage{amsthm}
\newtheorem{thm}{Theorem}[section]
\newtheorem{prop}[thm]{Proposition}
\newtheorem{lem}[thm]{Lemma}
\newtheorem{defi}[thm]{Definition}
\newtheorem{cor}[thm]{Corollary}
\newtheorem{rem}{Remark}[section]
\begin{document}
\title{The volume growth of hyperk\"ahler manifolds of type $A_{\infty}$}
\author{Kota Hattori}
\date{}
\maketitle
{\begin{center}
{\it Graduate School of Mathematical Sciences, University of Tokyo\\
3-8-1 Komaba, Meguro, Tokyo 153-8914, Japan\\
kthatto@ms.u-tokyo.ac.jp}
\end{center}}
\maketitle
{\abstract We study the volume growth of hyperk\"ahler manifolds of type $A_{\infty}$ constructed by Anderson-Kronheimer-LeBrun \cite{AKL} and Goto \cite{G1}. These are noncompact complete $4$-dimensional hyperk\"ahler manifolds of infinite topological type. These manifolds have the same topology but the hyperk\"ahler metrics are depends on the choice of parameters. By taking a certain parameter, we show that there exists a hyperk\"ahler manifold of type $A_{\infty}$ whose volume growth is $r^{\alpha}$ for each $3 <\alpha <4$.}
\section{Introduction}
\quad A hyperk\"ahler manifold is, by definition, a Riemannian manifold $(X,g)$ of real dimension $4n$ equipped with three complex structures $I_1,I_2,I_3$ satisfying the quaternionic relations $I_1^2 = I_2^2 = I_3^2 = I_1I_2I_3 = -{\rm id}$ with respect to all of which the metric $g$ is K\"ahlerian. Then the holonomy group of $g$ is a subgroup of $Sp(n)$ and $g$ is Ricci-flat.\\
\quad 
In this paper we focus on the volume growth of hyperk\"ahler metrics. The notion of the volume growth are considered for a Riemannian manifold $(X,g)$. We denote by $V_g(p_0,r)$ the volume of the ball $B_g(p_0,r)\subset X$ of radius $r$ centered at $p_0\in X$. Then we say that the volume growth of $g$ is $f(r)$ if the condition
\begin{eqnarray}
0 < \liminf_{r\to +\infty}\frac{V_g(p_0,r)}{f(r)} \le \limsup_{r\to +\infty}\frac{V_g(p_0,r)}{f(r)} < +\infty.\nonumber
\end{eqnarray}
holds for some $p_0\in X$. From Bishop-Gromov comparison theorem \cite{BC}\cite{GLP}, the above condition is independent of $p_0\in X$ if $g$ has the nonnegative Ricci curvature.\\
\quad 
For instance, the volume growth of Euclidean space $\mathbb{R}^4$ is $r^4$ and the volume growth of $\mathbb{R}^3\times S^1$ with the flat metric is $r^3$. These are trivial examples and there are also nontrivial examples such as ALE hyperk\"ahler metrics constructed in \cite{EH}\cite{GH}\cite{K} whose volume growth is $r^4$, and multi-Taub-NUT metrics \cite{H}\cite{NTU}\cite{T} whose volume growth is $r^3$.\\
\quad 
Thus there are several examples of complete hyperk\"ahler manifolds whose volume growth is $r^k$ for positive integers $k$. On the other hand, there exist complete Ricci-flat K\"ahler manifolds of complex dimension $n$ whose volume growth are $r^{(2n)/(n+1)}$ \cite{BK1}\cite{BK2}\cite{TY1}.\\
\quad 
In this paper we will show that there is a family of complete hyperk\"ahler manifolds of real dimension $4$ whose volume growth is less than $r^4$ and more than $r^3$.\\
\quad 
Anderson, Kronheimer and LeBrun constructed the hyperk\"ahler manifolds of type $A_{\infty}$ by Gibbons Hawking ansatz in \cite{AKL}, which are $4$-dimensional noncompact complete hyperk\"ahler manifolds of infinite topological type. Here infinite topological type means that the homology groups are infinitely generated. The same metrics were constructed by hyperk\"ahler quotinet method due to Goto in \cite{G1}. Each of the metrics in \cite{AKL} is constructed from an element of
\begin{eqnarray}
({\rm Im}\mathbb{H})_0^{\mathbb{Z}}:=\{\lambda = (\lambda_n)_{n\in\mathbb{Z}}\in({\rm Im}\mathbb{H})^{\mathbb{Z}};\ \sum_{n\in \mathbb{Z}}\frac{1}{1 + |\lambda_n|}<+\infty\},\nonumber
\end{eqnarray}
where $\mathbb{H}$ is the quaternions and ${\rm Im}\mathbb{H}$ the $3$-dimensional subspace formed by the purely imaginary quaternions. We denote by $(X_{\lambda},g_{\lambda})$ the hyperk\"ahler metric of type $A_{\infty}$ constructed from $\lambda=(\lambda_n)_{n\in\mathbb{Z}}\in ({\rm Im}\mathbb{H})_0^{\mathbb{Z}}$. The purpose of this paper is studying the asymptotic behavior of $V_{g_{\lambda}}(p_0,r)$ for some $p_0\in X_{\lambda}$, and observe how the volume growth of $g_{\lambda}$ depends on the choice of $\lambda$. The main result is described as follows.
\begin{thm}
For each $\lambda\in ({\rm Im}\mathbb{H})_0^{\mathbb{Z}}$ and $p_0\in X_{\lambda}$, the function $V_{g_{\lambda}}(p_0,r)$ satisfies
\begin{eqnarray}
0 < \liminf_{r\to +\infty}\frac{V_{g_{\lambda}}(p_0,r)}{r^2\tau_{\lambda}^{-1}(r^2)} \le \limsup_{r\to +\infty}\frac{V_{g_{\lambda}}(p_0,r)}{r^2\tau_{\lambda}^{-1}(r^2)} < +\infty,\nonumber
\end{eqnarray}
where the fuction $\tau_{\lambda}:\mathbb{R}_{\ge 0}\to\mathbb{R}_{\ge 0}$ is defined by
\begin{eqnarray}
\tau_{\lambda}(R):=\sum_{n\in\mathbb{Z}}\frac{R^2}{R + |\lambda_n|}\nonumber
\end{eqnarray}
for $R\ge 0$.
\end{thm}
Applying Theorem 1.1 to some $\lambda\in ({\rm Im}\mathbb{H})_0^{\mathbb{Z}}$, we can find hyperk\"aler manifolds whose volume growth is given as follows.
\begin{thm}
$(1)$ There is a hyperk\"ahler manifold $(X_{\lambda},g_{\lambda})$ for each $3 <\alpha <4$ which satisfies
\begin{eqnarray}
0 < \liminf_{r\to +\infty}\frac{V_{g_{\lambda}}(p_0,r)}{r^{\alpha}} \le \limsup_{r\to +\infty}\frac{V_{g_{\lambda}}(p_0,r)}{r^{\alpha}} < +\infty.\nonumber
\end{eqnarray}
$(2)$ There is a hyperk\"ahler manifold $(X_{\lambda},g_{\lambda})$ which satisfies
\begin{eqnarray}
\lim_{r\to +\infty}\frac{V_{g_{\lambda}}(r)}{r^4} = 0,\quad \lim_{r\to +\infty}\frac{V_{g_{\lambda}}(r)}{r^{\alpha}} = +\infty\nonumber
\end{eqnarray}
for any $\alpha <4$.
\end{thm}

Next we denote by $g_{\lambda}^{(s)}$ the Taub-NUT deformation of $g_{\lambda}$ where $s>0$ is the parameter of deformations. Then the volume growth of $g_{\lambda}^{(s)}$ is given by the following.
\begin{thm}
Let $\lambda\in ({\rm Im}\mathbb{H})_0^{\mathbb{Z}}$ and $s>0$. Then the volume growth of hyperk\"ahler metric $g_{\lambda}^{(s)}$ satisfies
\begin{eqnarray}
\lim_{r\to +\infty}\frac{V_{g_{\lambda}^{(s)}}(r)}{r^3} = \frac{8\pi^2}{3\sqrt{s}}.\nonumber
\end{eqnarray}
\end{thm}

In Section 2, we review the construction and the basic properties of hyperk\"ahler manifolds of type $A_{\infty}$. Although there are two constructions by Gibbons-Hawking ansatz and hyperk\"ahler quotient method, we adopt the latter way according to \cite{G1} since the argument in Section 3 is due to Goto's construction. Then we obtain the hyperk\"ahler manifold $(X_{\lambda},g_{\lambda})$ as a hyperk\"ahler quotient, and there are an $S^1$-action on $X_{\lambda}$ preserving the hyperk\"ahler structure and the hyperk\"ahler moment map $\mu_{\lambda}:X_{\lambda}\to {\rm Im}\mathbb{H}$.\\
\quad For the proof of Theorem 1.1, we need upper and lower bounds for the function $V_{g_{\lambda}}(p_0,r)$, which are discussed in Sections 3 and 4, respectively. In Section 3 the upper bound for $V_{g_{\lambda}}(p_0,r)$ will be obtained as follows. We fix $p_0\in X_{\lambda}$ to be $\mu_{\lambda}(p_0)=0\in{\rm Im}\mathbb{H}$. Then we obtain two inequalities
\begin{eqnarray}
{\rm vol}_{g_{\lambda}}(\mu_{\lambda}^{-1}(B_R)) &\le& P_+ R^2\varphi_{\lambda}(R),\nonumber\\
d_{g_{\lambda}}(p_0,p)^2 &\ge& Q_-|\mu_{\lambda}(p)|\cdot\varphi_{\lambda}(|\mu_{\lambda}(p)|),\nonumber
\end{eqnarray}
for each $p\in X_{\lambda}$, where ${\rm vol}_{g_{\lambda}}$ is the Riemannian measure, $d_{g_{\lambda}}$ is the Riemannian distance, $B_R:=\{\zeta\in {\rm Im}\mathbb{H};\ |\zeta |<R\}$, and $P_+$, $Q_-$ are positive constants. The former inequality is obtained from the hyperk\"ahler construction, and the latter one is obtained from Gibbons Hawking ansatz. By combining these inequalities we have the upper bound for $V_{g_{\lambda}}(p_0,r)$.\\
\quad Next we discuss the lower bound for $V_{g_{\lambda}}(p_0,r)$ in Section 4. If we try to bound the function $V_{g_{\lambda}}(p_0,r)$ from below in a similar method as in Section 3, then we need to show the inequality
\begin{eqnarray}
d_{g_{\lambda}}(p_0,p)^2 &\le& Q_+|\mu_{\lambda}(p)|\cdot\varphi_{\lambda}(|\mu_{\lambda}(p)|),\nonumber
\end{eqnarray}
which seems to be hard to show for any $p\in X_{\lambda}$. But it does not mean that the author can give a counterexample to the inequality. In Section 4, we take an open subset of $B_{g_{\lambda}}(p_0,r)$ which consists of points $p\in B_{g_{\lambda}}(p_0,r)$ satisfying the inequality $d_{g_{\lambda}}(p_0,p)^2 \le Q_+|\mu_{\lambda}(p)|\cdot\varphi_{\lambda}(|\mu_{\lambda}(p)|)$, and compute the volume of the subset, which is the lower bound for $V_{g_{\lambda}}(p_0,r)$.\\
\quad From two types of estimates obtained in Sections 3 and 4, we prove Theorem 1.1 in Section 5.\\
\quad We compute the volume growth of some examples concretely in Section 6, and obtain Theorem 1.2. Section 7 is devoted to studying the volume growth of the Taub-NUT deformations of $(X_{\lambda},g_{\lambda})$.\\
\quad \\
{\it Acknowledgment.}\ The author would like to thank Professor Hiroshi Konno for several advice on this paper. The author was supported by Grant-in-Aid for JSPS Fellows ($20\cdot 7215$).\\

\section{Hyperk\"ahler manifolds of type $A_{\infty}$}
\quad In this section, we review the construction of hyperk\"ahler manifolds of type $A_{\infty}$ according to \cite{G1}. Although they can be constructed by Gibbons-Hawking ansatz \cite{AKL}, we need hyperk\"ahler quotient construction in \cite{G1} for arguments in Section 3. Before the construction, we start with some basic definitions.
\begin{defi}
{\rm Let $(X,g)$ be a Riemannian manifold of dimension $4n$ and $I_1,I_2,I_3$ be complex structures on $X$ compatible with $g$. Then $(X,g,I_1,I_2,I_3)$ is a hyperk\"ahler manifold if $(I_1,I_2,I_3)$ satisfies the quaternionic relations $I_1^2 = I_2^2 = I_3^2 = I_1I_2I_3 = -{\rm id}$ and each $\omega_i:=g(I_i\cdot,\cdot)$ is closed for $i=1,2,3$.}
\end{defi}
Let $\mathbb{H}=\mathbb{R}\oplus\mathbb{R}i\oplus\mathbb{R}j\oplus\mathbb{R}k=\mathbb{C}\oplus\mathbb{C}j$ be the quaternions and ${\rm Im}\mathbb{H} = \mathbb{R}i\oplus\mathbb{R}j\oplus\mathbb{R}k$ the $3$-dimensional subspace formed by the purely imaginary quaternions. We can combine three fundamental $2$-forms into a single ${\rm Im}\mathbb{H}$-valued $2$-form $\omega:=\omega_1 i + \omega_2 j+ \omega_3 k$. Then it turns out that three complex structures and the metric are determined by the $2$-form $\omega$. Hence we call $\omega$ the hyperk\"ahler structure on $X$ instead of $(g,I_1,I_2,I_3)$.\\
\quad Let $G$ be a Lie group acting on a manifold $X$. Then each element $\xi$ of the Lie algebra of $G$ generates a vector field $\xi^*\in\mathcal{X}(X)$ defined by $\xi^*_x:=\frac{d}{dt}|_{t=0}x\exp (t\xi)\in T_xX$ for $x\in X$.
\begin{defi}
{\rm Suppose that the $n$-dimensional torus $T^n$ with Lie algebra $t^n$ acts smoothly on a hyperk\"ahler manifold $(X,\omega)$ preserving the hyperk\"ahler structure $\omega$. Then the hyperk\"ahler moment map $\mu :X\to {\rm Im}\mathbb{H}\otimes (t^n)^*$ for the action of $T^n$ on $X$ is defined by two properties (i) $\mu$ is $T^n$-invariant, (ii) $\langle d\mu_x(V),\xi\rangle = \omega(\xi^*_x,V_x)\in {\rm Im}\mathbb{H}$ for all $x\in X$ and $V\in T_xX$. Here $\langle ,\rangle : {\rm Im}\mathbb{H}\otimes (t^n)^*\otimes t^n\to {\rm Im}\mathbb{H}$ is the natural pairing induced from the contraction.}
\end{defi}
\quad Next we review the construction of hyperk\"ahler manifolds of type $A_{\infty}$. Let $S^{\mathbb{Z}}$ be the set of maps from the integers $\mathbb{Z}$ to a set $S$. Then each $x\in S^{\mathbb{Z}}$ is expressed in the sequence $x=(x_n)_{n\in\mathbb{Z}}$ where $x_n\in S$. Then we define a Hilbert space $M$ in the infinite sequences of the quaternions by
\begin{eqnarray}
M:=\{v\in\mathbb{H}^{\mathbb{Z}};\ \| v\|_M^2<+\infty\},\nonumber
\end{eqnarray}
where the Hilbert metric is given by
\begin{eqnarray}
\langle u,v\rangle_M:=\sum_{n\in \mathbb{Z}}u_n\bar{v}_n,\quad \| v\|^2_M:=\langle v,v\rangle_M\nonumber
\end{eqnarray}
for $u,v\in \mathbb{H}^{\mathbb{Z}}$.\\
\quad Put $\mathbb{H}^{\mathbb{Z}}_0:=\{\Lambda\in\mathbb{H}^{\mathbb{Z}};\ \sum_{n\in \mathbb{Z}}(1+|\Lambda_n|^2)^{-1}<+\infty\}$. Then for each $\Lambda\in\mathbb{H}^{\mathbb{Z}}_0$, we have the following Hilbert manifolds
\begin{eqnarray}
M_{\Lambda} &:=& \Lambda + M = \{\Lambda + v ;\ v\in M\},\nonumber\\
U_{\Lambda} &:=& \{g=(g_n)_{n\in\mathbb{Z}}\in (S^1)^{\mathbb{Z}};\ \|1_{\mathbb{Z}}-g\|^2_{\Lambda}<+\infty\},\nonumber\\
\mathbf{u}_{\Lambda} &:=& \{\xi=(\xi_n)_{n\in\mathbb{Z}}\in \mathbb{R}^{\mathbb{Z}};\ \|\xi\|^2_{\Lambda}<+\infty\},\nonumber\\
G_{\Lambda} &:=& \{g=(g_n)_{n\in\mathbb{Z}}\in U_{\Lambda};\ \prod_{n\in\mathbb{Z}}g_n=1\},\nonumber\\
\mathbf{g}_{\Lambda} &:=& \{\xi=(\xi_n)_{n\in\mathbb{Z}}\in \mathbf{u}_{\Lambda};\ \sum_{n\in\mathbb{Z}}\xi_n=0\},\nonumber
\end{eqnarray}
where
\begin{eqnarray}
\langle \xi,\eta\rangle_{\Lambda}:=\sum_{n\in \mathbb{Z}}(1+|\Lambda_n|^2)\xi_n\bar{\eta}_n, \quad \|\xi\|^2_{\Lambda}:=\langle \xi,\xi\rangle_{\Lambda}\nonumber
\end{eqnarray}
for $\xi,\eta\in\mathbb{C}^{\mathbb{Z}}$. Here, $1_{\mathbb{Z}}\in (S^1)^{\mathbb{Z}}$ is the constant map $(1_{\mathbb{Z}})_n:=1$. The convergence of $\prod_{n\in\mathbb{Z}}g_n$ and $\sum_{n\in\mathbb{Z}}\xi_n$ follows from the condition $\sum_{n\in \mathbb{Z}}(1+\|\Lambda_n\|^2)^{-1}<+\infty$. Then $G_{\Lambda}$ is a Hilbert Lie group whose Lie algebra is $\mathbf{g}_{\Lambda}$. We can define a right action of $G_{\Lambda}$ on $M_{\Lambda}$ by $xg:=(x_ng_n)_{n\in \mathbb{Z}}$ for $x\in M_{\Lambda}$, $g\in G_{\Lambda}$. Here the product of $x_n$ and $g_n$ is given by regarding $S^1$ as the subset of $\mathbb{H}$ by the natural injections $S^1\subset\mathbb{C}\subset\mathbb{H}$.\\
\quad Since $M$ is a left $\mathbb{H}$-module defined by $h v:=(h v_n)_{n\in \mathbb{Z}}$ for $v\in M$ and $h\in\mathbb{H}$, $M$ has a hyperk\"ahler structure given by the left multiplication of $i,j,k$ and inner product $\langle ,\rangle_M$ on $M$. We denote by $I_1,I_2,I_3\in {\rm End}(M)$ complex structures induced by the left multiplication by $i,j,k$, respectively. Thus $M_{\Lambda}$ is an infinite dimensional hyperk\"ahler manifold and the action of $G_{\Lambda}$ on $M_{\Lambda}$ preserves the hyperk\"ahler structure. Then we define a map $\hat{\mu}_{\Lambda}:M_{\Lambda}\rightarrow {\rm Im}\mathbb{H}\otimes\mathbf{g}_{\Lambda}^*$ by
\begin{eqnarray}
\langle \hat{\mu}_{\Lambda}(x),\xi\rangle :=\sum_{n\in \mathbb{Z}}(x_ni\bar{x}_n - \Lambda_ni\bar{\Lambda}_n)\xi_n\ \in {\rm Im}\mathbb{H}\nonumber
\end{eqnarray}
for $x \in M_{\Lambda}$, $\xi\in\mathbf{g}_{\Lambda}$. This $\hat{\mu}_{\Lambda}$ is the hyperk\"ahler moment map for the action of infinite dimensional torus $G_{\Lambda}$ on $M_{\Lambda}$.\\
\quad Since $\hat{\mu}_{\Lambda}$ is $G_{\Lambda}$-invariant, then $G_{\Lambda}$ acts on the inverse image $\hat{\mu}_{\Lambda}^{-1}(0)$. Hence we obtain the quotient space $\hat{\mu}_{\Lambda}^{-1}(0)/G_{\Lambda}$ which is called hyperk\"ahler quotient. In general, there is no guarantee that $\hat{\mu}_{\Lambda}^{-1}(0)/G_{\Lambda}$ is a smooth manifold for every $\Lambda\in\mathbb{H}_0^\mathbb{Z}$. For the smoothness of $\hat{\mu}_{\Lambda}^{-1}(0)/G_{\Lambda}$ we need the following condition.
\begin{defi}
{\rm An element $\Lambda\in\mathbb{H}_0^\mathbb{Z}$ is generic if $\Lambda_ni\bar{\Lambda}_n - \Lambda_mi\bar{\Lambda}_m \neq 0$ for all distinct $n,m\in\mathbb{Z}$.}
\end{defi}
\begin{prop}
For a generic $\Lambda\in\mathbb{H}_0^\mathbb{Z}$, the Lie group $G_{\Lambda}$ acts freely on $\hat{\mu}_{\Lambda}^{-1}(0)$.
\end{prop}
\begin{proof}
Let $e_l\in\mathbf{g}_{\Lambda}$ be defined by $e_l:=(e_{l,n})_{n\in \mathbb{Z}}$, where
\[ e_{l,n}=
\left\{
\begin{array}{ccc}
1 & (n=l), \\
0 & (n\neq l).
\end{array}
\right.
\]
Since $\mathbf{g}_{\Lambda}$ is generated by elements $e_l-e_m$, then $x\in M_{\Lambda}$ satisfies $\hat{\mu}_{\Lambda}(x)=0$ if and only if $\langle \hat{\mu}_{\Lambda}(x),e_l-e_m\rangle  = 0$ for all $l,m\in\mathbb{Z}$. Hence the value of $x_ni\bar{x}_n - \Lambda_ni\bar{\Lambda}_n$ is independent of $n\in\mathbb{Z}$ for $x\in \hat{\mu}_{\Lambda}^{-1}(0)$.\\
\quad Assume that there is a pair of $x\in \hat{\mu}_{\Lambda}^{-1}(0)$ and $g\in G_{\Lambda}$ satisfies $xg=x$. If $x_n\neq 0$ for any $n\in\mathbb{Z}$, then $g=1$. Therefore we may assume $x_s=0$ for some $s\in\mathbb{Z}$. Then we have $x_ni\bar{x}_n - \Lambda_ni\bar{\Lambda}_n = - \Lambda_si\bar{\Lambda}_s$ for all $n\in\mathbb{Z}$, which implies
\begin{eqnarray}
x_ni\bar{x}_n = \Lambda_ni\bar{\Lambda}_n - \Lambda_si\bar{\Lambda}_s \neq 0\nonumber
\end{eqnarray}
for $n\neq s$. Since $x_n=0$ if and only if $x_ni\bar{x}_n=0$, we have $x_n\neq 0$ if $n\neq s$. Thus we have shown $g_n=1$ if $n\neq s$, and also $g_s$ should be $1$ from the condition $\prod_{n\in\mathbb{Z}}g_n=1$.
\end{proof}
\begin{thm}[\cite{G1}]
If $\Lambda\in\mathbb{H}_0^\mathbb{Z}$ is generic, then $\hat{\mu}_{\Lambda}^{-1}(0)/G_{\Lambda}$ is a smooth manifold of dimension $4$, and the hyperk\"ahler structure on $M_{\Lambda}$ induces a hyperk\"ahler structure $\hat{\omega}_{\Lambda}$ on $\hat{\mu}_{\Lambda}^{-1}(0)/G_{\Lambda}$.
\end{thm}
\begin{rem}
{\rm In \cite{G1}, the above theorem is proved in the case of
\[ \Lambda_n=
\left\{
\begin{array}{ccc}
ni & (n\ge 0), \\
nk & (n<0).
\end{array}
\right.
\]
But it is easy to show the theorem for the case of any generic $\Lambda\in\mathbb{H}_0^\mathbb{Z}$.}
\end{rem}
Take $\Lambda\in\mathbb{H}_0^\mathbb{Z}$ and $(e^{i\theta_n})_{n\in\mathbb{Z}}\in (S^1)^{\mathbb{Z}}$ and put $\Lambda' = (\Lambda_n e^{i\theta_n})_{n\in\mathbb{Z}}$. Then there is a canonical isomorphism of hyperk\"ahler structure between $\hat{\mu}_{\Lambda}^{-1}(0)/G_{\Lambda}$ and $\hat{\mu}_{\Lambda'}^{-1}(0)/G_{\Lambda'}$ as follows. Define a map $\hat{F}:\hat{\mu}_{\Lambda}^{-1}(0)\to\hat{\mu}_{\Lambda'}^{-1}(0)$ by 
\begin{eqnarray}
\hat{F}((x_n)_{n\in \mathbb{Z}}):=(x_ne^{i\theta_n})_{n\in \mathbb{Z}}.\nonumber
\end{eqnarray}
Since $\hat{F}$ is equivariant with respect to $G_{\Lambda}$ and $G_{\Lambda'}$-actions, we obtain $F:\hat{\mu}_{\Lambda}^{-1}(0)/G_{\Lambda}\to\hat{\mu}_{\Lambda'}^{-1}(0)/G_{\Lambda'}$, which preserves hyperk\"ahler structures $\hat{\omega}_{\Lambda}$ and $\hat{\omega}_{\Lambda'}$. In this case we have $\Lambda_n i \bar{\Lambda}_n = \Lambda'_n i \bar{\Lambda}'_n$ for all $n\in \mathbb{Z}$. Converesely, if we take $\Lambda,\Lambda'\in\mathbb{H}_0^\mathbb{Z}$ to be $(\Lambda_n i \bar{\Lambda}_n)_{n\in \mathbb{Z}} = (\Lambda'_n i \bar{\Lambda}'_n)_{n\in \mathbb{Z}}$, then there is $(e^{i\theta_n})_{n\in\mathbb{Z}}\in (S^1)^{\mathbb{Z}}$ such that $\Lambda'_n = \Lambda_n e^{i\theta_n}$. Thus $\hat{\mu}_{\Lambda}^{-1}(0)/G_{\Lambda}$ and $\hat{\mu}_{\Lambda'}^{-1}(0)/G_{\Lambda'}$ are isomorphic as hyperk\"ahler manifolds if $(\Lambda_n i \bar{\Lambda}_n)_{n\in \mathbb{Z}} = (\Lambda'_n i \bar{\Lambda}'_n)_{n\in \mathbb{Z}}$. Therefore we may put $(X_{\lambda},\omega_{\lambda}):=(\hat{\mu}_{\Lambda}^{-1}(0)/G_{\Lambda},\hat{\omega}_{\Lambda})$ for
\begin{eqnarray}
\lambda\in ({\rm Im}\mathbb{H})_0^{\mathbb{Z}} &:=& \{(\Lambda_n i \bar{\Lambda}_n)_{n\in \mathbb{Z}}\in({\rm Im}\mathbb{H})^{\mathbb{Z}};\ \Lambda\in\mathbb{H}_0^{\mathbb{Z}}\}\nonumber\\
&=& \{\lambda\in({\rm Im}\mathbb{H})^{\mathbb{Z}};\ \sum_{n\in \mathbb{Z}}\frac{1}{1 + |\lambda_n|}<+\infty\}\nonumber
\end{eqnarray}
and $\Lambda\in\hat{\nu}^{-1}(\lambda)$. Then the condition for $X_{\lambda}$ to be smooth is written as follows.
\begin{defi}
{\rm An element $\lambda\in({\rm Im}\mathbb{H})_0^\mathbb{Z}$ is generic if $\lambda_n - \lambda_m \neq 0$ for all distinct $n,m\in\mathbb{Z}$.}
\end{defi}
Then Theorem 2.5 implies that $(X_{\lambda},\omega_{\lambda})$ is a smooth hyperk\"ahler manifold for each generic $\lambda\in({\rm Im}\mathbb{H})_0^\mathbb{Z}$.\\
\quad We denote by $g_{\lambda}$ the hyperk\"ahler metric induced from $\omega_{\lambda}$. 
\begin{thm}[\cite{G1}]
Let $(X_{\lambda},g_{\lambda})$ be as above. Then the Riemannian metric $g_{\lambda}$ is complete.
\end{thm}
We denote by $[x]\in \hat{\mu}_{\Lambda}^{-1}(0)/G_{\Lambda}$ the equivalence class represented by $x\in\hat{\mu}_{\Lambda}^{-1}(0)$. Then by putting
\begin{eqnarray}
[x]g:=[(\cdots,x_{-1},x_0g,x_1,\cdots,x_n,\cdots)]\nonumber
\end{eqnarray}
for $x=(x_n)_{n\in \mathbb{Z}}=(\cdots,x_{-1},x_0,x_1,\cdots,x_n,\cdots)\in\hat{\mu}_{\Lambda}^{-1}(0)$ and $g\in S^1$, we have the action of $S^1$ on $X_{\lambda}$. Then the hyperk\"ahler moment map $\mu_{\lambda}:X_{\lambda}\rightarrow {\rm Im}\mathbb{H}=\mathbb{R}^3$ is given by
\begin{eqnarray}
\mu_{\lambda}([x]):=x_0i\bar{x}_0 - \lambda_0 \in {\rm Im}\mathbb{H}.\nonumber
\end{eqnarray}
Since $x\in\hat{\mu}_{\Lambda}^{-1}(0)$, we have $x_0i\bar{x}_0 - \lambda_0 = x_ni\bar{x}_n - \lambda_n$, where we identify $(t^1)^*$ with $\mathbb{R}$ by taking a generator of $t^1$.\\
\quad If we put $Stab([x]):=\{g\in S^1;\ [x]g=[x]\}$ for $[x]\in X_{\lambda}$, then it is obvious that $Stab([x])=\{1\}$ if and only if $x_n\neq 0$ for any $n\in\mathbb{Z}$, otherwise $Stab([x])=S^1$. Hence we have a principal $S^1$-bundle $\mu_{\lambda}\big|_{X_{\lambda}^*}:X_{\lambda}^*\rightarrow Y_{\lambda}$ where
\begin{eqnarray}
X_{\lambda}^* &:=& \{[x]\in X_{\lambda};\ x_n\neq 0\ {\rm for\ all\ }n\in\mathbb{Z}\},\nonumber\\
Y_{\lambda} &:=& {\rm Im}\mathbb{H}\backslash \{- \lambda_n;\ n\in\mathbb{Z}\}\nonumber.
\end{eqnarray}
\begin{defi}
{\rm An $S^1$-action on a $4$-dimensional hyperk\"ahler manifold \\
$(X,\omega)$ is tri-Hamiltonian if the action preserves $\omega$ and there exists a hyperk\"ahler moment map $\mu :X\to {\rm Im}\mathbb{H}$ for the action of $S^1$.}
\end{defi}
\begin{thm}[\cite{GH}]
There exists a canonical one-to-one correspondence between the followings;
(i) a hyperk\"ahler manifold of real dimension $4$ with free tri-Hamiltonian $S^1$-action,
(ii) a principal $S^1$-bundle $\mu:X\rightarrow Y$ where $Y$ is an open subset of $\mathbb{R}^3$, and an $S^1$-connection $A$ on $X$ and a positive valued harmonic function $\Phi$ on $Y$ such that $\frac{dA}{2\sqrt{-1}} = \mu^*(*d\Phi)$. Here $*$ is the Hodge star operator with respect to the Euclidean metric on $\mathbb{R}^3$.
\end{thm}
Here we see a sketch of the proof, and the details can be seen in \cite{G1}.\\
\quad Let $(X,\omega)$ be a hyperk\"ahler manifold of dimension $4$ with a free $S^1$-action preserving $\omega$, and $\mu:X\to {\rm Im}\mathbb{H}$ the hyperk\"ahler moment map. Then $(Y,\Phi,A)$ is given by the followings. Let $Y$ be the image $\mu (X)$. Then the function $\Phi:Y\to\mathbb{R}$ is defined by
\begin{eqnarray}
\frac{1}{\Phi(\mu (x))}:=4 g_x(\xi^*_x,\xi^*_x)\nonumber
\end{eqnarray}
for $x\in X$, where $g$ is the hyperk\"ahler metric and $\xi:=\sqrt{-1}$ is a generator of Lie algebre of the Lie group $S^1$. For each $x\in X$, we denote by $V_x\subset T_xX$ the subspace spanned by $\xi^*_x$. Then the $S^1$-connection $A$ on $X$ is defined by the horizontal distribution $(H_x)_{x\in X}$ where $H_x\subset T_xX$ is the orthogonal complement of $V_x$.\\
\quad Conversely, let $\mu=(\mu_1,\mu_2,\mu_3):X\rightarrow Y\subset {\rm Im}\mathbb{H}$ be a principal $S^1$-bundle and $(\Phi,A)$ a pair consisting a positive valued harmonic function and a connection of the $S^1$-bundle with $\frac{dA}{2\sqrt{-1}} = \mu^*(*d\Phi)$. Then the hyperk\"ahler structure $\omega=(\omega_1,\omega_2,\omega_3)\in\Omega^2(X)\otimes {\rm Im}\mathbb{H}$ is defined by
\begin{eqnarray}
\omega_1:=d\mu_1\wedge \frac{A}{2\sqrt{-1}} + \mu^*\Phi d\mu_2\wedge d\mu_3,\nonumber\\
\omega_2:=d\mu_2\wedge \frac{A}{2\sqrt{-1}} + \mu^*\Phi d\mu_3\wedge d\mu_1,\nonumber\\
\omega_3:=d\mu_3\wedge \frac{A}{2\sqrt{-1}} + \mu^*\Phi d\mu_1\wedge d\mu_2.\nonumber
\end{eqnarray}
Then it follows from the condition $\frac{dA}{2\sqrt{-1}} = \mu^*(*d\Phi)$ that $\omega$ is closed.\\
\quad Theorem 2.9 gives the positive valued harmonic function $\Phi_{\lambda}$ on $Y_{\lambda}$ and $S^1$-connection $A_{\lambda}$ on $X_{\lambda}^*$. It is known in \cite{G1} that $\Phi_{\lambda}$ is given by
\begin{eqnarray}
\Phi_{\lambda}(\zeta) = \frac{1}{4}\sum_{n\in\mathbb{Z}}\frac{1}{|\zeta + \lambda_n|}\nonumber
\end{eqnarray}
for $\zeta\in Y_{\lambda}$.\\
\quad Let $(X,\omega)$ be a hyperk\"ahler manifold satisfying the condition (i) of Theorem 2.9, and $(Y,\Phi,A)$ be what corresponds to $(X,\omega)$ satisfying the condition (ii). Denote by $g$ the hyperk\"ahler metric of $\omega$ and let ${\rm vol}_g(B)$ be the volume of a subset $B\subset X$.
\begin{lem}
Let $U\subset {\rm Im}\mathbb{H}$ be an open set. Then we have the following formula
\begin{eqnarray}
{\rm vol}_{g}(\mu^{-1}(U)) = \pi\int_{\zeta\in U}\Phi(\zeta)d\zeta_1d\zeta_2d\zeta_3,\nonumber
\end{eqnarray}
where $\zeta = (\zeta_1,\zeta_2,\zeta_3)\in {\rm Im}\mathbb{H}$ is the Cartesian coordinate.
\end{lem}
\begin{proof}
It suffices to show the assertion for all open set $U\subset Y$.\\
\quad First of all we suppose that the principal $S^1$-bundle $\mu:\mu^{-1}(U)\to U$ is trivial. Then we can take a $C^{\infty}$ trivialization $\sigma:U\to\mu^{-1}(U)$ and define $C^{\infty}$ map $t:\mu^{-1}(U)\to\mathbb{R}/2\pi\mathbb{Z}$ by $t(\sigma(\zeta)e^{i\theta}):=\theta$ for $\zeta\in U$ and $\theta\in\mathbb{R}/2\pi\mathbb{Z}$ and obtain a local coordinate $(t,\mu_1,\mu_2,\mu_3)$ on $\mu^{-1}(U)$. Since we may write $dt = -\sqrt{-1}A + \sum_{l=1}^3a_ld\mu^l$ for some $a_1,a_2,a_3\in C^{\infty}(\mu^{-1}(U))$, the volume form ${\rm vol}_g$ is given by
\begin{eqnarray}
{\rm vol}_g &=& \frac{\mu^*\Phi}{2}(-\sqrt{-1}) d\mu_1 \wedge d\mu_2 \wedge d\mu_3 \wedge A\nonumber\\
&=& \frac{\mu^*\Phi}{2} d\mu_1 \wedge d\mu_2 \wedge d\mu_3\wedge dt.\nonumber
\end{eqnarray}
Thus we have
\begin{eqnarray}
{\rm vol}_g(\mu^{-1}(U)) &=& \int_{\mu^{-1}(U)}{\rm vol}_g\nonumber\\
&=& \int_{\mu^{-1}(U)}\frac{\mu^*\Phi}{2}d\mu_1 \wedge d\mu_2 \wedge d\mu_3\wedge dt\nonumber\\
&=& \pi\int_{\zeta\in U}\Phi(\zeta)d\zeta_1d\zeta_2d\zeta_3.\nonumber
\end{eqnarray}
For a general $U\subset Y$, we take open sets $U_1,U_2,\cdots,U_m\subset Y$ such that each principal $S^1$-bundle $\mu^{-1}(U_{\alpha})\to U_{\alpha}$ is trivial and $\coprod_{\alpha=1}^m U_{\alpha}\subset U \subset \overline{\coprod_{\alpha=1}^m U_{\alpha}}$. Then we have
\begin{eqnarray}
{\rm vol}_g(\mu^{-1}(U)) &=& \sum_{\alpha=1}^m {\rm vol}_g(\mu^{-1}(U_{\alpha}))\nonumber\\
&=& \sum_{\alpha=1}^m\pi\int_{\zeta\in U_{\alpha}}\Phi(\zeta)d\zeta_1d\zeta_2d\zeta_3\nonumber\\
&=& \pi\int_{\zeta\in U}\Phi(\zeta)d\zeta_1d\zeta_2d\zeta_3.\nonumber
\end{eqnarray}
\end{proof}
\section{The upper bound for the volume growth}
\quad For a Riemannian manifold $(X,g)$ and a point $p_0\in X$, we denote by $V_g(p_0,r)$ the volume of a ball $B_g(p_0,r):=\{p\in X;\ d_g(p_0,p)<r\}$ with respect to the Riemannian distance $d_g$. In this section, we will evaluate the upper bound for $V_{g_{\lambda}}(p_0,r)$ for the hyperk\"ahler manifold $(X_{\lambda},g_{\lambda})$.\\
\quad In Section 2, we supposed $\lambda\in ({\rm Im}\mathbb{H})_0^\mathbb{Z}$ to be generic for the smoothness of $(X_{\lambda},g_{\lambda})$. But the function $V_{g_{\lambda}}(p_0,r)$ is determined only by the Riemannian measure ${\rm vol}_{g_{\lambda}}$ and the Riemannian distance $d_{g_{\lambda}}$. Even if $\lambda\in ({\rm Im}\mathbb{H})_0^\mathbb{Z}$ is taken not to be generic, ${\rm vol}_{g_{\lambda}}$ and $d_{g_{\lambda}}$ can be extended to $X_{\lambda}$ naturally. Hence we take $\lambda\in ({\rm Im}\mathbb{H})_0^\mathbb{Z}$ not necessary to be generic from now on.\\
\quad Fix $p_0\in X_{\lambda}$ to be $\mu_{\lambda}(p_0)= - \lambda_0$. We may assume $\lambda_0=0$ since the hyperk\"ahler quotient constructed from $(\lambda_n-\lambda_0)_{n\in\mathbb{Z}}$ is isometric to $(X_{\lambda},g_{\lambda})$. For each $R > 0$, put
\begin{eqnarray}
\varphi_{\lambda}(R):=\sum_{n\in\mathbb{Z}}\frac{R}{R + |\lambda_n|},\nonumber
\end{eqnarray}
and $\varphi_{\lambda}(0):=\lim_{R\to +0}\varphi_{\lambda}(R)$.
\begin{prop}
We have the following inequality
\begin{eqnarray}
d_{g_{\lambda}}(p_0,p)^2\ge Q_-|\mu_{\lambda}(p)|\cdot\varphi_{\lambda}(|\mu_{\lambda}(p)|)\nonumber
\end{eqnarray}
for any $p\in X_{\lambda}$, where $Q_-=1/8$.
\end{prop}
\begin{proof}
Take $\Lambda = (\Lambda_n)_{n\in\mathbb{Z}}\in\mathbb{H}_0^\mathbb{Z}$ to be $\Lambda_n i \bar{\Lambda}_n = \lambda_n$ for all $n\in\mathbb{Z}$. We fix $x\in\hat{\mu}_{\Lambda}^{-1}(0)$ such that $[x] = p$. If we regard $\hat{\mu}_{\Lambda}^{-1}(0)$ as the infinite dimensional Riemannian submanifold of $M_{\Lambda}$, then the quotient map $\hat{\mu}_{\Lambda}^{-1}(0)\to X_{\lambda} = \hat{\mu}_{\Lambda}^{-1}(0)/G_{\Lambda}$ is a Riemannian submersion and $G_{\Lambda}$ acts on $\hat{\mu}_{\Lambda}^{-1}(0)$ as an isometry. Then the horizontal lift of the geodesic from $p_0$ to $p$ has the same length as $d_{g_{\lambda}}(p_0,p)$. Since the Riemannian distance between $\Lambda$ and $x'\in\hat{\mu}_{\Lambda}^{-1}(0)$ is larger than $\|\Lambda - x'\|_{M}$, then we have
\begin{eqnarray}
d_{g_{\lambda}}(p_0,p) \ge \inf_{\sigma\in G_{\Lambda}}\|\Lambda - x\sigma\|_{M}.\nonumber
\end{eqnarray}
If we put $\Lambda=(\alpha_n + \beta_nj)_{n\in\mathbb{Z}}$, $x=(z_n + w_nj)_{n\in\mathbb{Z}}$ and $\sigma = (e^{i\theta_n})_{n\in\mathbb{Z}}$, then
\begin{eqnarray}
\|\Lambda - x\sigma\|^2_{M} = \sum_{n\in\mathbb{Z}}\big(|\alpha_n - z_n e^{i\theta_n}|^2 + |\beta_n - w_n e^{-i\theta_n}|^2\big).\nonumber
\end{eqnarray}
If the function $f_n(t):=|\alpha_n - z_n e^{it}|^2 + |\beta_n - w_n e^{-it}|^2$ attains its minimum at $\theta_n$ for each $n\in\mathbb{Z}$, then $\sigma = (e^{i\theta_n})_{n\in\mathbb{Z}}$ satisfies$\|\Lambda - x\sigma\|_{M} \le \inf_{\tau\in G_{\Lambda}}\|\Lambda - x\tau\|_{M}$.\\
\quad From the equations
\begin{eqnarray}
\lambda_n = \Lambda_n i \overline{\Lambda}_n = i(|\alpha_n|^2 - |\beta_n|^2) - 2\alpha_n\beta_n k,\nonumber\\
\zeta + \lambda_n = x_n i \overline{x}_n = i(|z_n|^2 - |w_n|^2) - 2z_n w_n k,\nonumber
\end{eqnarray}
we have
\begin{eqnarray}
f_n(t) = |\lambda_n| + |\zeta + \lambda_n| - 2Re\{(\alpha_n\bar{z}_n + \bar{\beta}_nw_n)e^{-it}\},\nonumber\\
|\alpha_n\bar{z}_n + \bar{\beta}_nw_n|^2 = \frac{1}{2}\big(|\lambda_n||\zeta + \lambda_n| + \langle \lambda_n,\zeta + \lambda_n\rangle_{\mathbb{R}^3}\big),\nonumber
\end{eqnarray}
where $\langle ,\rangle_{\mathbb{R}^3}$ is the standard inner product on ${\rm Im}\mathbb{H} = \mathbb{R}^3$.\\
\quad Suppose $n\neq 0$. If $\alpha_n\bar{z}_n + \bar{\beta}_nw_n\neq 0$, then we put $e^{i\theta_n}:=\frac{\alpha_n\bar{z}_n + \bar{\beta}_nw_n}{|\alpha_n\bar{z}_n + \bar{\beta}_nw_n|}$. If $\alpha_n\bar{z}_n + \bar{\beta}_nw_n = 0$, we may put $\theta_n:=0$ since $f_n(t)$ is constant.\\
\quad Let $S:=\{n\in\mathbb{Z};\ \alpha_n\bar{z}_n + \bar{\beta}_nw_n = 0\}$. For each $n\notin S$, we have
\begin{eqnarray}
 |\alpha_n - z_n e^{i\theta_n}|^2 + |\beta_n - w_n e^{-i\theta_n}|^2 &=& \frac{|\zeta|^2}{|\lambda_n| + |\zeta + \lambda_n| + 2|\alpha_n\bar{z}_n + \bar{\beta}_nw_n|} \nonumber\\
&\le& \frac{|\zeta|^2}{|\lambda_n|}.\nonumber
\end{eqnarray}
Hence we deduce
\begin{eqnarray}
\sum_{n\in\mathbb{Z}}|\Lambda_n - x_n e^{i\theta_n}|^2 \le \sum_{n\in S}|\Lambda_n - x_n|^2 + \sum_{n\notin S}\frac{|\zeta|^2}{|\lambda_n|} < +\infty.\nonumber
\end{eqnarray}
Thus we obtain $\sum_{n\in\mathbb{Z}\backslash\{0\}}|\Lambda_n|^2|1 - e^{i\theta_n}|^2 < +\infty$, which ensures the convergence of $\Pi_{n\in\mathbb{Z}\backslash\{0\}}e^{i\theta_n}$.\\
\quad It follows from $\lambda_0 = 0$ that $\alpha_0 = \beta_0 = 0$. Then $f_0(t)$ is constant. Since $f_0(t)$ attains its minimum at any $\theta_0$, then we may put $e^{i\theta_0}:=\Pi_{n\in\mathbb{Z}\backslash\{0\}}e^{-i\theta_n}$. Thus the function $\|\Lambda - x\tau\|_M$ attains its minimum at $\tau = \sigma = (e^{i\theta_n})_{n\in\mathbb{Z}}\in G_{\Lambda}$. For this $\sigma$, we have
\begin{eqnarray}
\|\Lambda - x\sigma\|^2_M &=& \sum_{n\in\mathbb{Z}}(|\lambda_n| + |\zeta + \lambda_n| - \sqrt{2}\sqrt{|\lambda_n||\zeta + \lambda_n| + \langle \lambda_n,\zeta + \lambda_n\rangle_{\mathbb{R}^3}})\nonumber\\
&=& \sum_{n\in\mathbb{Z}}\frac{|\zeta|^2}{|\lambda_n| + |\zeta + \lambda_n| + \sqrt{2}\sqrt{|\lambda_n||\zeta + \lambda_n| + \langle \lambda_n,\zeta + \lambda_n\rangle_{\mathbb{R}^3}}}\nonumber\\
&\ge&  \sum_{n\in\mathbb{Z}}\frac{|\zeta|^2}{8|\lambda_n| + 2|\zeta|} \ge \frac{1}{8}|\zeta|\cdot\varphi_{\lambda}(|\zeta|).\nonumber
\end{eqnarray}
Then the assertion is obtained by putting $Q_-=\frac{1}{8}$.
\end{proof}
Put $B_R:=\{\zeta\in {\rm Im}\mathbb{H};\ |\zeta|<R\}$. Next we discuss the upper bound for the volume of $\mu_{\lambda}^{-1}(B_R)$.
\begin{lem}
We define two functions $N_{\lambda}(R)$ and $\psi_{\lambda}(R)$ by
\begin{eqnarray}
N_{\lambda}(R) &:=& \{n\in\mathbb{Z};\ |\lambda_n|\le R\},\nonumber\\
\psi_{\lambda}(R) &:=& \sharp N_{\lambda}(R) + \sum_{n\notin N_{\lambda}(R)}\frac{1}{|\lambda_n|}R\nonumber
\end{eqnarray}
for $R\ge 0$. Then we have
\begin{eqnarray}
\varphi_{\lambda}(R)\le\psi_{\lambda}(R)\le 2\varphi_{\lambda}(R).\nonumber
\end{eqnarray}
\end{lem}
\begin{proof}
We define two functions $p(x)$ and $q(x)$ by
\begin{eqnarray}
p(x):=\frac{x}{1+x},\quad q(x):=\min\{1,x\}\nonumber
\end{eqnarray}
for $x\ge 0$. Then the inequalities
\begin{eqnarray}
p(x)\le q(x)\le 2p(x)\nonumber
\end{eqnarray}
hold for each $x\ge 0$. Therefore the assertion follows from
\begin{eqnarray}
\varphi_{\lambda}(R) = \sum_{n\in\mathbb{Z}}p\bigg(\frac{R}{|\lambda_n|}\bigg),\quad\psi_{\lambda}(R) = \sum_{n\in\mathbb{Z}}q\bigg(\frac{R}{|\lambda_n|}\bigg).\nonumber
\end{eqnarray}
\end{proof}
\begin{lem}
For $\alpha\ge 1$ and $R\ge 0$, we have $\varphi_{\lambda}(\alpha R)\le\alpha\varphi_{\lambda}(R)$ and $\psi_{\lambda}(\alpha R)\le\alpha\psi_{\lambda}(R)$.
\end{lem}
\begin{proof}
Since $\frac{\varphi_{\lambda}(R)}{R}$ is strictly decreasing for $R$, we have
\begin{eqnarray}
\varphi_{\lambda}(\alpha R) = \frac{\varphi_{\lambda}(\alpha R)}{\alpha R}\alpha R \le \frac{\varphi_{\lambda}(R)}{R}\alpha R = \alpha\varphi_{\lambda}(R)\nonumber
\end{eqnarray}
for $\alpha\ge 1$ and $R\ge 0$. It follows from the same argument that $\psi_{\lambda}(\alpha R)\le\alpha\psi_{\lambda}(R)$.
\end{proof}
\begin{prop}
There is a constant $P_+>0$ independent of $\lambda$ and $R$ such that
\begin{eqnarray}
{\rm vol}_{g_{\lambda}}(\mu_{\lambda}^{-1}(B_R))\le P_+ R^2\varphi_{\lambda}(R)\nonumber
\end{eqnarray}
for all $R\ge 0$.
\end{prop}
\begin{proof}
From Lemma 2.10, the volume of $\mu_{\lambda}^{-1}(B_R)$ is given by
\begin{eqnarray}
{\rm vol}_{g_{\lambda}}(\mu_{\lambda}^{-1}(B_R)) &=& \pi\int_{\zeta\in B_R}\Phi_{\lambda}(\zeta)d\zeta_1d\zeta_2d\zeta_3\nonumber\\
&=& \frac{\pi}{4}\sum_{n\in\mathbb{Z}}\int_{\zeta\in B_R}\frac{1}{|\zeta + \lambda_n|}d\zeta_1d\zeta_2d\zeta_3.\nonumber
\end{eqnarray}
Let $(r\ge 0,\Theta)$ be the polar coordinate over ${\rm Im}\mathbb{H} = \mathbb{R}^3$, where $\Theta$ is a coordinate on $S^2=\partial B_1$. If $n\in N_{\lambda}(R)$, then the change of variables $\zeta' = \zeta + \lambda_n$ gives
\begin{eqnarray}
\int_{\zeta\in B_R}\frac{1}{|\zeta + \lambda_n|}d\zeta_1d\zeta_2d\zeta_3 &\le& \int_{\zeta'\in B_{R + |\lambda_n|}}\frac{1}{|\zeta'|}d\zeta'_1d\zeta'_2d\zeta'_3\nonumber\\
&=& 4\pi\int_0^{R + |\lambda_n|}rdr\nonumber\\
&=& 2\pi(R + |\lambda_n|)^2\le 8\pi R^2.\nonumber
\end{eqnarray}
If $n\notin N_{\lambda}(R)$, the mean value property of harmonic functions gives
\begin{eqnarray}
\int_{\zeta\in B_R}\frac{1}{|\zeta + \lambda_n|}d\zeta_1d\zeta_2d\zeta_3 &=& \frac{4\pi R^3}{3}\frac{1}{|\lambda_n|},\nonumber
\end{eqnarray}
since $\frac{1}{|\zeta + \lambda_n|}$ is harmonic on $B_R$. Hence the upper bound for ${\rm vol}_{g_{\lambda}}(\mu_{\lambda}^{-1}(B_R))$ is given by
\begin{eqnarray}
{\rm vol}_{g_{\lambda}}(\mu_{\lambda}^{-1}(B_R)) &\le& 2\pi^2\sharp N_{\lambda}(R)\cdot R^2 + \frac{\pi^2}{3}\sum_{n\notin N_{\lambda}(R)}\frac{R}{|\lambda_n|}\cdot R^2\nonumber\\
&\le& 2\pi^2\psi_{\lambda}(R) R^2 \le 4\pi^2\varphi_{\lambda}(R) R^2.\nonumber
\end{eqnarray}
Then we have the assertion by putting $P_+ := 4\pi^2$.
\end{proof}
Let $\theta_{\lambda,C}(R):=C\varphi_{\lambda}(R) R^2$, $\tau_{\lambda,C}(R):=C\varphi_{\lambda}(R) R$ for $C>0$, and $\tau_{\lambda,C}^{-1}:\mathbb{R}_{\ge 0}\to\mathbb{R}_{\ge 0}$ be the inverse function of $\tau_{\lambda,C}$.
\begin{prop}
Let $P_+$ and $Q_-$ be as in Proposition 3.1 and 3.4. Then the inequality
\begin{eqnarray}
V_{g_{\lambda}}(p_0,r)\le\theta_{\lambda,P_+}\circ\tau_{\lambda,Q_-}^{-1}(r^2)\nonumber
\end{eqnarray}
holds for all $r\ge 0$.
\end{prop}
\begin{proof}
Since $p\in B_{g_{\lambda}}(p_0,\sqrt{\tau_{\lambda,Q_-}(R)})$ satisfies $d_{g_{\lambda}}(p_0,p)^2\le\tau_{\lambda,Q_-}(R)$, then from Proposition 3.1 we obtain\begin{eqnarray}
\tau_{\lambda,Q_-}(|\mu_{\lambda}(p)|) = Q_-|\mu_{\lambda}(p)|\varphi(|\mu_{\lambda}(p)|) \le \tau_{\lambda,Q_-}(R).\nonumber
\end{eqnarray}
Since the function $\tau_{\lambda,Q_-}(R)$ is strictly increasing, we have $|\mu_{\lambda}(p)|\le R$. Then we have $B_{g_{\lambda}}(p_0,\sqrt{\tau_{\lambda,Q_-}(R)})\subset\mu^{-1}_{\lambda}(B_R)$. Then it follows from Proposition 3.4 that we have
\begin{eqnarray}
{\rm vol}_{g_{\lambda}}(B_{g_{\lambda}}(p_0,\sqrt{\tau_{\lambda,Q_-}(R)})) \le {\rm vol}_{g_{\lambda}}(\mu_{\lambda}^{-1}(B_R)) \le P_+R^2\varphi(R) = \theta_{\lambda,P_+}(R).\nonumber
\end{eqnarray}
Set $r=\sqrt{\tau_{\lambda,Q_-}(R)}$. Then substituting for $R$ the inverse function $\tau_{\lambda,Q_-}^{-1}(r^2)$, we obtain the result.
\end{proof}

\section{The lower bound for the volume growth}
\quad In the previous section we obtained the upper bound for $V_{g_{\lambda}}(p_0,r)$. The purpose of this section is to obtain the inequality
\begin{eqnarray}
\theta_{\lambda,P_-}\circ\tau_{\lambda,Q_+}^{-1}(r^2) \le V_{g_{\lambda}}(p_0,r),\nonumber
\end{eqnarray}
where $P_-,Q_+ >0$ are constants independent of $\lambda$ and $r$. In similar consideration as in Section 3, it seems that the following two inequalities should be shown,
\begin{eqnarray}
{\rm vol}_{g_{\lambda}}(\mu_{\lambda}^{-1}(B_R)) &\ge& P_- R^2\varphi_{\lambda}(R),\\
d_{g_{\lambda}}(p_0,p)^2 &\le& Q_+|\mu_{\lambda}(p)|\cdot\varphi_{\lambda}(|\mu_{\lambda}(p)|).
\end{eqnarray}
But it seems to be hard to show the inequality $(2)$ for the author. Accordingly we shall show two modified inequalities; one is a stronger inequality than $(1)$, and the other is weaker than $(2)$. First of all we consider the former estimate.\\
\quad Let $U\subset S^2$ be a measurable set and put
\begin{eqnarray}
B_{R,U}:=\{t\zeta\in {\rm Im}\mathbb{H};\ 0\le t\le R,\ \zeta\in U\}.\nonumber
\end{eqnarray}
We denote by $m_{S^2}$ the measure induced from the Riemannian metric with constant curvature over $S^2$ whose total measure is given by $m_{S^2}(S^2) = 4\pi$. First of all we consider the lower bound for ${\rm vol}_{g_{\lambda}}(\mu_{\lambda}^{-1}(B_{R,U}))$.
\begin{prop}
There is a constant $C_->0$ independent of $\lambda$, $R$ and $U\subset S^2$, which satisfies
\begin{eqnarray}
{\rm vol}_{g_{\lambda}}(\mu_{\lambda}^{-1}(B_{R,U}))\ge C_-m_{S^2}(U)R^2\cdot\varphi_{\lambda}(R).\nonumber
\end{eqnarray}
\end{prop}
\begin{proof}
From the triangle inequality, we have
\begin{eqnarray}
{\rm vol}_{g_{\lambda}}(\mu_{\lambda}^{-1}(B_{R,U})) &\ge& \frac{\pi}{4}\sum_{n\in\mathbb{Z}}\int_{\zeta\in B_{R,U}}\frac{1}{|\zeta| + |\lambda_n|}d\zeta_1d\zeta_2d\zeta_3\nonumber\\
&=& \frac{\pi}{4}\sum_{n\in\mathbb{Z}}\int_{r\in[0,R]}\int_{\Theta\in U}\frac{1}{r + |\lambda_n|}r^2drdm_{S^2}\nonumber\\
&=& \frac{\pi m_{S^2}(U)}{4}\sum_{n\in\mathbb{Z}}\int_{r\in[0,R]}\frac{1}{r + |\lambda_n|}r^2dr.\nonumber
\end{eqnarray}
If $n\in N_{\lambda}(R)$, then
\begin{eqnarray}
\int_{r\in[0,R]}\frac{1}{r + |\lambda_n|}r^2dr \ge \int_{r\in[0,R]}\frac{1}{r + R}r^2dr = (\log 2-\frac{1}{2})R^2.\nonumber
\end{eqnarray}
If $n\notin N_{\lambda}(R)$, then
\begin{eqnarray}
\int_{r\in[0,R]}\frac{1}{r + |\lambda_n|}r^2dr &=& \int_{r\in[0,R]}\frac{1}{r + |\lambda_n|}r^2dr\nonumber\\
&=& |\lambda_n|^2\bigg\{ -\frac{R}{|\lambda_n|} + \frac{1}{2}\bigg(\frac{R}{|\lambda_n|}\bigg)^2 + \log\bigg(1 + \frac{R}{|\lambda_n|}\bigg)\bigg\}.\nonumber
\end{eqnarray}
Since the inequality
\begin{eqnarray}
\log(1 + x)\ge x - \frac{1}{2}x^2 + \frac{1}{3}x^3 - \frac{1}{4}x^4\nonumber
\end{eqnarray}
holds for $x\ge 0$, then we have
\begin{eqnarray}
\int_{r\in[0,R]}\frac{1}{r + |\lambda_n|}r^2dr &\ge& |\lambda_n|^2\bigg\{ \frac{1}{3}\bigg(\frac{R}{|\lambda_n|}\bigg)^3 - \frac{1}{4}\bigg(\frac{R}{|\lambda_n|}\bigg)^4\bigg\}\nonumber\\
&\ge& \frac{1}{12}\frac{R^3}{|\lambda_n|}.\nonumber
\end{eqnarray}
By taking $\frac{4C_-}{\pi}:=\min\{\log 2-\frac{1}{2},\frac{1}{12}\} >0$, we have
\begin{eqnarray}
{\rm vol}_{g_{\lambda}}(\mu_{\lambda}^{-1}(B_{R,U})) \ge C_-m_{S^2}(U)R^2\cdot\psi_{\lambda}(R) \ge  C_-m_{S^2}(U)R^2\cdot\varphi_{\lambda}(R).\nonumber
\end{eqnarray}
\end{proof}
Next we need the upper bound for $d_{g_{\lambda}}(p_0,p)$. If we can calculate the length of a piecewise smooth path from $p_0$ to $p$, then the length is larger than $d_{g_{\lambda}}(p_0,p)$. Here we take the path as follows.\\
\quad Put $\zeta = \mu_{\lambda}(p)$. For $|\zeta|\le 1$, we define $\gamma_p$ to be the geodesic from $p_0$ to $p$. If $|\zeta| > 1$, we define a function $\delta:[1,|\zeta|]\to {\rm Im}\mathbb{H}$ by $\delta(t):=t\frac{\zeta}{|\zeta|}$ and take the horizontal lift $\tilde{\delta}:[1,|\zeta|]\to X_{\lambda}$ of $\delta$ with respect to the connection $A_{\lambda}$ such that $\tilde{\delta}(|\zeta|)=p$. If there exists a point $t\in [1,|\zeta|]$ such that $\delta(t)=-\lambda_n$ for some $n\in\mathbb{Z}$, then $\tilde{\delta}$ is not always smooth but continuous and piecewise smooth. Then we obtain a path $\gamma_p$ by connecting the geodesic from $p_0$ to $\tilde{\delta}(1)$ and $\tilde{\delta}$.\\
\quad The length of $\gamma_p$ is given by $d_{g_{\lambda}}(p_0,\tilde{\delta}(1)) + l_{\lambda}(\mu_{\lambda}(p))$ where $l_{\lambda}:{\rm Im}\mathbb{H}\to\mathbb{R}$ is defined by
\begin{eqnarray}
l_{\lambda}(\zeta):=\int_1^{|\zeta|}\sqrt{\Phi_{\lambda}\bigg(t\frac{\zeta}{|\zeta|}\bigg)}dt\nonumber
\end{eqnarray}
for $|\zeta|\ge 1$, and $l_{\lambda}(\zeta):=0$ for $|\zeta|\le 1$.\\
\quad If there exists a point $t\in [1,|\zeta|]$ such that $\delta(t)=-\lambda_n$, then $\sqrt{\Phi_{\lambda}\bigg(t\frac{\zeta}{|\zeta|}\bigg)}$ is not continuous at $t = |\lambda_n|$. But the integral $\int_1^{|\zeta|}\sqrt{\Phi_{\lambda}\bigg(t\frac{\zeta}{|\zeta|}\bigg)}dt$ is well-defined since the integral
\begin{eqnarray}
\int_{b-\varepsilon}^{b+\varepsilon}\sqrt{\frac{1}{a|t-b|} + h(t)}dt\nonumber
\end{eqnarray}
is a finite value for any constant $a>0$, $b\in\mathbb{R}$ and smooth function $h(t)$.\\
\quad Now we have $d_{g_{\lambda}}(p_0,p)\le L_{\lambda} + l_{\lambda}(\mu_{\lambda}(p))$, where
\begin{eqnarray}
L_{\lambda}:=\sup_{p\in\mu_{\lambda}^{-1}(\bar{B}_1)}d_{g_{\lambda}}(p_0,p) < +\infty.\nonumber
\end{eqnarray}
\begin{prop}
There is a constant $C_+>0$ independent of $\lambda$ and $R$ which satisfies
\begin{eqnarray}
\int_{\Theta\in \partial B_1}l_{\lambda}(R\Theta)dm_{S^2} \le 4\pi\sqrt{C_+R\varphi_{\lambda}(R)}.\nonumber
\end{eqnarray}
\end{prop}
\begin{proof}
We may suppose $R\ge 1$, since the left hand side of the assertion is equal to $0$ if $R<1$. The definition of $l_{\lambda}$ gives
\begin{eqnarray}
\int_{\Theta\in \partial B_1}l_{\lambda}(R\Theta)dm_{S^2} &\le& \int_{\Theta\in \partial B_1}\int_1^R\sqrt{\Phi_{\lambda}(t\Theta)}dtdm_{S^2}\nonumber\\
&=& \int_{\zeta\in B_R\backslash B_1}\frac{\sqrt{\Phi_{\lambda}(\zeta)}}{|\zeta|^2}d\zeta,
\end{eqnarray}
where $d\zeta = d\zeta_1d\zeta_2d\zeta_3$.\\
\quad Take $m\in\mathbb{Z}_{\ge 0}$ to be $2^m\le R<2^{m+1}$. Then the Cauchy-Schwarz inequality gives
\begin{eqnarray}
(3) &=& \sum_{l=0}^{m-1}\int_{\zeta\in B_{2^{l+1}}\backslash B_{2^l}}\frac{\sqrt{\Phi_{\lambda}(\zeta)}}{|\zeta|^2}d\zeta + \int_{\zeta\in B_R\backslash B_{2^m}}\frac{\sqrt{\Phi_{\lambda}(\zeta)}}{|\zeta|^2}d\zeta\nonumber\\
&\le& \sum_{l=0}^{m-1}\sqrt{\int_{(r,\Theta)\in B_{2^{l+1}}\backslash B_{2^l}}\frac{r^2drd\Theta}{r^4}}\sqrt{\int_{\zeta\in B_{2^{l+1}}\backslash B_{2^l}}\Phi_{\lambda}(\zeta)d\zeta}\nonumber\\
&\ & \quad  + \sqrt{\int_{(r,\Theta)\in B_R\backslash B_{2^m}}\frac{r^2drd\Theta}{r^4}}\sqrt{\int_{\zeta\in B_R\backslash B_{2^m}}\Phi_{\lambda}(\zeta)d\zeta}.\nonumber
\end{eqnarray}
From Proposition 3.4, the inequalities
\begin{eqnarray}
\int_{\zeta\in B_t\backslash B_{t'}}\Phi_{\lambda}(\zeta)d\zeta \le \int_{\zeta\in B_t}\Phi_{\lambda}(\zeta)d\zeta \le \frac{P_+}{\pi}t^2\varphi_{\lambda}(t)\nonumber
\end{eqnarray}
hold for any $0\le t'\le t$. Therefore the assertion follows from
\begin{eqnarray}
(3) &\le& \sqrt{4\pi}\sum_{l=0}^{m-1}\sqrt{\frac{1}{2^l} - \frac{1}{2^{l + 1}}}\sqrt{\frac{P_+}{\pi}2^{2(l + 1)}\varphi_{\lambda}(2^{l + 1})}\nonumber\\
&\ & \quad  + \sqrt{4\pi}\sqrt{\frac{1}{2^m} - \frac{1}{R}}\sqrt{\frac{P_+}{\pi}R^2\varphi_{\lambda}(R)}\nonumber\\
&\le& 2\sum_{l=0}^{m-1}\sqrt{2}^{l+1}\sqrt{P_+\varphi_{\lambda}(R)} + 2\sqrt{P_+R\varphi_{\lambda}(R)}\nonumber\\
&\le& 2(3 + \sqrt{2})\sqrt{P_+R\varphi_{\lambda}(R)}\nonumber
\end{eqnarray}
by putting $C_+=(\frac{3 + \sqrt{2}}{2\pi})^2P_+$.
\end{proof}
Since $R\varphi_{\lambda}(R)$ diverges to $+\infty$ for $R\to +\infty$, there is a constant $R_0>0$ which satisfies
\begin{eqnarray}
4\pi L_{\lambda} + \int_{\Theta\in S^2}l_{\lambda}(R\Theta)dm_{S^2} \le 4\pi\sqrt{2C_+R\varphi_{\lambda}(R)}
\end{eqnarray}
for all $R\ge R_0$. Now we put
\begin{eqnarray}
U_{R,T}:=\{\Theta\in S^2;\ L_{\lambda} + l_{\lambda}(R\Theta)\le \sqrt{TR\varphi_{\lambda}(R)}\}\nonumber
\end{eqnarray}
for $R,T>0$. For a fixed $R>0$, we introduce a function $F(\Theta)$ on $S^2$ by
\begin{eqnarray}
F(\Theta):=L_{\lambda} + l_{\lambda}(R\Theta).\nonumber
\end{eqnarray}
Then $U_{R,T}$ is the inverse image $F^{-1}([0,\sqrt{TR\varphi_{\lambda}(R)}])$, hence $U_{R,T}$ is a measurable set.
\begin{lem}
There exists a sufficiently large $R_0 > 0$ and we have
\begin{eqnarray}
m_{S^2}(U_{R,T}) \ge 4\pi\frac{ \sqrt{T} - \sqrt{2C_+}}{\sqrt{T} - \sqrt{Q_-}}\nonumber
\end{eqnarray}
for $R\ge R_0$ and $T>2C_+$.
\end{lem}
\begin{proof}
From $(4)$, we have the upper bound
\begin{eqnarray}
4\pi\sqrt{2C_+R\varphi_{\lambda}(R)} \ge \int_{\Theta\in S^2}F(\Theta)dm_{S^2}.
\end{eqnarray}
Since $F(\Theta):=L_{\lambda} + l_{\lambda}(R\Theta)\ge d_{g_{\lambda}}(p_0,p)$ and $\mu_{\lambda}(p) = \zeta =R\Theta$, Lemma 3.1 yields the lower bound for $F(\Theta)$,
\begin{eqnarray}
F(\Theta) \ge \sqrt{Q_-R\varphi_{\lambda}(R)}.\nonumber
\end{eqnarray}
On the complement $U_{R,T}^c$ of $U_{R,T}$ in $S^2$, we have $F(\Theta) \ge \sqrt{TR\varphi_{\lambda}(R)}$. Then we have
\begin{eqnarray}
\int_{\Theta\in S^2}F(\Theta)dm_{S^2} &=& \int_{\Theta\in U_{R,T}}F(\Theta)dm_{S^2} + \int_{\Theta\in U_{R,T}^c}F(\Theta)dm_{S^2} \nonumber\\
&\ge& m_{S^2}(U_{R,T})\sqrt{Q_-R\varphi_{\lambda}(R)}\nonumber\\
&\ & \quad + (4\pi - m_{S^2}(U_{R,T}))\sqrt{TR\varphi_{\lambda}(R)}.
\end{eqnarray}
From inequalities $(5)$ and $(6)$, we have
\begin{eqnarray}
4\pi\sqrt{2C_+R\varphi_{\lambda}(R)} &\ge& m_{S^2}(U_{R,T})\sqrt{Q_-R\varphi_{\lambda}(R)}\nonumber\\
&\ & \quad + (4\pi - m_{S^2}(U_{R,T}))\sqrt{TR\varphi_{\lambda}(R)}.\nonumber
\end{eqnarray}
Thus we have
\begin{eqnarray}
(\sqrt{T} - \sqrt{Q_-})m_{S^2}(U_{R,T})\sqrt{R\varphi_{\lambda}(R)} \ge 4\pi(\sqrt{T} - \sqrt{2C_+})\sqrt{R\varphi_{\lambda}(R)}.\nonumber
\end{eqnarray}
Since $T > 2C_+ > Q_-$, we obtain
\begin{eqnarray}
m_{S^2}(U_{R,T}) \ge 4\pi\frac{\sqrt{T} - \sqrt{2C_+}}{\sqrt{T} - \sqrt{Q_-}}.\nonumber
\end{eqnarray}
\end{proof}
\begin{lem}
For each $R\ge 0$ and $T>0$, $\mu_{\lambda}^{-1}(B_{R,U_{R,T}})$ is a subset of $B_{g_{\lambda}}(p_0, \sqrt{\tau_{\lambda,T}(R)})$.
\end{lem}
\begin{proof}
First of all we take $p\in \mu_{\lambda}^{-1}(B_{R,U_{R,T}})$ such that $|\mu_{\lambda}(p)|> 1$. Since $\frac{\mu_{\lambda}(p)}{|\mu_{\lambda}(p)|}$ is an element of $U_{R,T}$, we have
\begin{eqnarray}
L_{\lambda} + l_{\lambda}(\mu_{\lambda}(p)) \le L_{\lambda} + l_{\lambda}\bigg(R\frac{\mu_{\lambda}(p)}{|\mu_{\lambda}(p)|}\bigg) \le \sqrt{TR\varphi_{\lambda}(R)}.\nonumber
\end{eqnarray}
from $R\ge |\mu_{\lambda}(p)| >1$. Then we obtain $d_{g_{\lambda}}(p_0,p)\le \sqrt{TR\varphi_{\lambda}(R)}$ from $d_{g_{\lambda}}(p_0,p)\le L_{\lambda} + l_{\lambda}(\mu_{\lambda}(p))$.\\
\quad If $p\in \mu_{\lambda}^{-1}(B_{R,U_{R,T}})$ is taken to be $|\mu_{\lambda}(p)|\le 1$, then we have the same conclusion as above since $l_{\lambda}(\mu_{\lambda}(p))=0$ in this case.
\end{proof}
Now we fix a constant $Q_+$ to be $Q_+>2C_+$ and put $m_0:=4\pi\frac{\sqrt{Q_+} - \sqrt{2C_+}}{\sqrt{Q_+} - \sqrt{Q_-}}$ and $P_-:=m_0C_-$.
\begin{prop}
Let $P_-,Q_+>0$ be as above. Then we have
\begin{eqnarray}
\liminf_{r\to\infty}\frac{V_{g_{\lambda}}(p_0,r)}{\theta_{\lambda,P_-}\circ\tau_{\lambda,Q_+}^{-1}(r^2)}>0.\nonumber
\end{eqnarray}
\end{prop}
\begin{proof}
Let $R\ge 0$. From Lemma 4.4, we have
\begin{eqnarray}
V_{g_{\lambda}}(p_0,\sqrt{\tau_{\lambda,Q_+}(R)})\ge {\rm vol}_{g_{\lambda}}(\mu_{\lambda}^{-1}(B_{R,U_{R,Q_+}})).\nonumber
\end{eqnarray}
Then Proposition 4.1 gives
\begin{eqnarray}
V_{g_{\lambda}}(p_0,\sqrt{\tau_{\lambda,Q_+}(R)}) &\ge& m_{S^2}(U_{R,Q_+})C_-R^2\varphi_{\lambda}(R)\nonumber\\
&\ge& m_0C_-R^2\varphi_{\lambda}(R) = \theta_{\lambda,P_-}(R)\nonumber
\end{eqnarray}
for $R\ge R_0$. Thus we have the assertion by putting $R=\tau_{\lambda,Q_+}^{-1}(r^2)$.
\end{proof}
\section{The volume growth}
\quad In Sections 3 and 4, we estimate $V_{g_{\lambda}}(p_0,r)$ from the above by $\theta_{\lambda,P_+}\circ\tau_{\lambda,Q_-}^{-1}(r^2)$ and from the bottom by $\theta_{\lambda,P_-}\circ\tau_{\lambda,Q_+}^{-1}(r^2)$. In this section we show that the asymptotic behavior of the functions $\theta_{\lambda,P_+}\circ\tau_{\lambda,Q_-}^{-1}(r^2)$ and $\theta_{\lambda,P_-}\circ\tau_{\lambda,Q_+}^{-1}(r^2)$ are equal up to constant, and prove the main results.\\
\quad The asymptotic behavior of $V_{g_{\lambda}}(p,r)$ is independent of the choice of $p\in X_{\lambda}$ from the next well-known fact.
\begin{prop}
Let $(X,g)$ be a connected Riemannian manifold of dimension $n$, whose Ricci curvature is nonnegative. Then we have
\begin{eqnarray}
\lim_{r\to +\infty}\frac{V_g(p_1,r)}{V_g(p_0,r)} = 1\nonumber
\end{eqnarray}
for any $p_0,p_1\in X$.
\end{prop}
\begin{proof}
From the Bishop-Gromov comparison inequality $\frac{V_g(p_1,r)}{r^n}$ is nonincreasing with respect to $r$. If we put $r_0:=d_g(p_0,p_1)$, then we have
\begin{eqnarray}
\frac{V_g(p_1,r)}{V_g(p_0,r)} &\le& \frac{V_g(p_0,r + r_0)}{V_g(p_0,r)}\nonumber\\&=& \frac{r^n}{V_g(p_0,r)}\frac{V_g(p_0,r + r_0)}{(r + r_0)^n}\frac{(r + r_0)^n}{r^n}\nonumber\\
&\le& \frac{r^n}{V_g(p_0,r)}\frac{V_g(p_0,r)}{r^n}\frac{(r + r_0)^n}{r^n}\nonumber\\
&=& \frac{(r + r_0)^n}{r^n}\to 1\quad (r\to +\infty).\nonumber
\end{eqnarray}
By considering the same argument after exchanging $p_0$ and $p_1$, we have the assertion.
\end{proof}
We denote by $C_+^0$ the set of the nondecreasing continuous functions from $\mathbb{R}_{\ge 0}$ to $\mathbb{R}_{\ge 0}$.
\begin{defi}
{\rm For $f_0,f_1\in C_+^0$, we define $f_0(r)\preceq_r f_1(r)$ if
\begin{eqnarray}
\limsup _{r\to +\infty}\frac{f_0(r)}{f_1(r)}<+\infty.\nonumber
\end{eqnarray}
}
\end{defi}
\begin{defi}
{\rm For $f_0,f_1\in C_+^0$, we define $f_0(r)\simeq_r f_1(r)$ if $f_0(r)\preceq_r f_1(r)$ and $f_1(r)\preceq_r f_0(r)$.}
\end{defi}
For a Riemannian manifold $(X,g)$ and a point $p_0\in X$, the function $V_g(p_0,r)$ is an element of $C_+^0$. If $(X,g)$ satisfies the assumption of Proposition 5.1, then the equivalence class of $V_g(p_0,r)$ with respect to $\simeq_r$ is independent of the choice of $p_0\in X$. Therefore we denoted by $V_g(r)$ the equivalence class of $V_g(p_0,r)$. For example, if we write $V_g(r)\preceq_r r^n$, then it implies that
\begin{eqnarray}
\limsup_{r\to +\infty}\frac{V_g(p_0,r)}{r^n}<+\infty\nonumber
\end{eqnarray}
for some (hence all) $p_0\in X$.\\
\quad The main purpose of this section is to look for the function in $C_+^0$ which is equivalent to $V_{g_{\lambda}}(r)$.
\begin{lem}
Let $S_+,S_-,T_+,T_->0$. Then $\theta_{\lambda,S_+}\circ\tau_{\lambda,T_-}^{-1}$ and $\theta_{\lambda,S_-}\circ\tau_{\lambda,T_+}^{-1}$ are elements of $C_+^0$ and we have $\theta_{\lambda,S_+}\circ\tau_{\lambda,T_-}^{-1}(r^2) \simeq_r \theta_{\lambda,S_-}\circ\tau_{\lambda,T_+}^{-1}(r^2)$.
\end{lem}
\begin{proof}
We may suppose $T_-\le T_+$ without loss of generality. Since $\tau_{\lambda,T_+}$ and $\tau_{\lambda,T_-}$ are continuous, strictly increasing and satisfy $\tau_{\lambda,T_+}(0) = \tau_{\lambda,T_-}(0) = 0$, then $\tau_{\lambda,T_+}^{-1}$ and $\tau_{\lambda,T_-}^{-1}$ are also the elements of $C_+^0$. Hence the composite functions $\theta_{\lambda,S_+}\circ\tau_{\lambda,T_-}^{-1}$ and $\theta_{\lambda,S_-}\circ\tau_{\lambda,T_+}^{-1}$ are also the elements of $C_+^0$.\\
\quad Next we show (i) $\theta_{\lambda,S_-}\circ\tau_{\lambda,T_+}^{-1}(r^2) \preceq_r \theta_{\lambda,S_+}\circ\tau_{\lambda,T_-}^{-1}(r^2)$ and (ii) $\theta_{\lambda,S_+}\circ\tau_{\lambda,T_-}^{-1}(r^2) \preceq_r \theta_{\lambda,S_-}\circ\tau_{\lambda,T_+}^{-1}(r^2)$.\\
(i) From $\tau_{\lambda,T_+}(R) \ge \tau_{\lambda,T_-}(R)$ and strictly increasingness of $\tau_{\lambda,T_{\pm}}$, then we have $\tau_{\lambda,T_+}^{-1}(r^2) \le \tau_{\lambda,T_-}^{-1}(r^2)$. Hence we obtain
\begin{eqnarray}
\frac{\theta_{\lambda,S_-}\circ\tau_{\lambda,T_+}^{-1}(r^2)}{\theta_{\lambda,S_+}\circ\tau_{\lambda,T_-}^{-1}(r^2)} \le 
\frac{\theta_{\lambda,S_-}\circ\tau_{\lambda,T_-}^{-1}(r^2)}{\theta_{\lambda,S_+}\circ\tau_{\lambda,T_-}^{-1}(r^2)} \le 
\frac{S_-}{S_+}.\nonumber
\end{eqnarray}
(ii) Put $R_{\pm}:=\tau_{\lambda,T_{\pm}}^{-1}(r^2)$. Then we have
\begin{eqnarray}
r^2 = T_+R_+\varphi_{\lambda}(R_+) = T_-R_-\varphi_{\lambda}(R_-).\nonumber
\end{eqnarray}
Since $\varphi_{\lambda}$ is nondecreasing, it holds
\begin{eqnarray}
R_-\varphi_{\lambda}(R_-) = \frac{T_+}{T_-}R_+\varphi_{\lambda}(R_+) \le \frac{T_+}{T_-}R_+\varphi_{\lambda}\bigg(\frac{T_+}{T_-}R_+\bigg)
\end{eqnarray}
Since the function $R\varphi_{\lambda}(R)$ is strictly increasing with respect to $R$, the expression $(7)$ gives $R_-\le\frac{T_+}{T_-}R_+$. Recall that $\varphi_{\lambda}$ satisfies $\varphi_{\lambda}(\alpha R) \le \alpha\varphi_{\lambda}(R)$ for $\alpha\ge 1$, which implies $\theta_{\lambda,S_+}(\alpha R) \le \alpha^3\theta_{\lambda,S_+}(R)$. Thus we have
\begin{eqnarray}
\frac{\theta_{\lambda,S_+}\circ\tau_{\lambda,T_-}^{-1}(r^2)}{\theta_{\lambda,S_-}\circ\tau_{\lambda,T_+}^{-1}(r^2)} = \frac{\theta_{\lambda,S_+}(R_-)}{\theta_{\lambda,S_-}(R_+)} \le \frac{\theta_{\lambda,S_+}\bigg(\frac{T_+}{T_-}R_+\bigg)}{\theta_{\lambda,S_-}(R_+)} &\le& \bigg(\frac{T_+}{T_-}\bigg)^3\frac{\theta_{\lambda,S_+}(R_+)}{\theta_{\lambda,S_-}(R_+)}\nonumber\\
&=& \bigg(\frac{T_+}{T_-}\bigg)^3\frac{S_+}{S_-}.\nonumber
\end{eqnarray}
\end{proof}
Put $\theta_{\lambda}:=\theta_{\lambda,1}$, $\tau_{\lambda}:=\tau_{\lambda,1}$. Then the main result in this paper is described as follows.
\begin{thm}
For each $\lambda\in ({\rm Im}\mathbb{H})_0^{\mathbb{Z}}$ and $p\in X_{\lambda}$, the function $V_{g_{\lambda}}(p,r)$ satisfies
\begin{eqnarray}
0 < \liminf_{r\to +\infty}\frac{V_{g_{\lambda}}(p,r)}{r^2\tau_{\lambda}^{-1}(r^2)} \le \limsup_{r\to +\infty}\frac{V_{g_{\lambda}}(p,r)}{r^2\tau_{\lambda}^{-1}(r^2)} < +\infty.\nonumber
\end{eqnarray}
\end{thm}
\begin{proof}
We have shown that
\begin{eqnarray}
\theta_{\lambda,P_-}\circ\tau_{\lambda,Q_+}^{-1}(r^2) \preceq_r V_{g_{\lambda}}(r) \preceq_r \theta_{\lambda,P_+}\circ\tau_{\lambda,Q_-}^{-1}(r^2)\nonumber
\end{eqnarray}
in Sections 3 and 4, and
\begin{eqnarray}
\theta_{\lambda,P_-}\circ\tau_{\lambda,Q_+}^{-1}(r^2) \simeq_r \theta_{\lambda,P_+}\circ\tau_{\lambda,Q_-}^{-1}(r^2) \simeq_r \theta_{\lambda}\circ\tau_{\lambda}^{-1}(r^2)\nonumber
\end{eqnarray}
in Lemma 5.4. Thus we have the assertion from
\begin{eqnarray}
\theta_{\lambda}\circ\tau_{\lambda}^{-1}(r^2) = \tau_{\lambda}^{-1}(r^2)\cdot\tau_{\lambda}(\tau_{\lambda}^{-1}(r^2)) = r^2\tau_{\lambda}^{-1}(r^2).\nonumber
\end{eqnarray}
\end{proof}
\begin{cor}
For $\lambda,\lambda'\in ({\rm Im}\mathbb{H})_0^{\mathbb{Z}}$, the condition $V_{g_{\lambda}}(r) \preceq_r V_{g_{\lambda'}}(r)$ is equivalent to $\varphi_{\lambda}(R) \preceq_R \varphi_{\lambda'}(R)$.
\end{cor}
\begin{proof}
The condition $V_{g_{\lambda}}(r) \preceq_r V_{g_{\lambda'}}(r)$ is equivalent to $\tau_{\lambda}^{-1}(r^2) \preceq_r \tau_{\lambda'}^{-1}(r^2)$ from Theorem 5.5.\\
\quad If we assume $\tau_{\lambda}^{-1}(r^2) \preceq_r \tau_{\lambda'}^{-1}(r^2)$ then there are constants $r_0\ge 0$ and $C\ge 1$ which satisfy $\frac{\tau_{\lambda}^{-1}(r^2)}{\tau_{\lambda'}^{-1}(r^2)} \le C$ for all $r\ge r_0$. Now we put $r^2:=\tau_{\lambda}(R)$ and $(r')^2:=\tau_{\lambda'}(R)$ for $R\ge\tau_{\lambda}^{-1}(r_0^2)$. Since we have
\begin{eqnarray}
R = \tau_{\lambda'}^{-1}((r')^2) = \tau_{\lambda}^{-1}(r^2) \le C\tau_{\lambda'}^{-1}(r^2),\nonumber
\end{eqnarray}
then
\begin{eqnarray}
(r')^2\le \tau_{\lambda'}(C\tau_{\lambda'}^{-1}(r^2)) \le C^2\tau_{\lambda'}(\tau_{\lambda'}^{-1}(r^2)) = C^2r^2\nonumber
\end{eqnarray}
is given by the monotonicity of $\tau_{\lambda'}$. Thus we obtain $\frac{\tau_{\lambda'}(R)}{\tau_{\lambda}(R)} \le C^2$ for all $R\ge \tau_{\lambda'}^{-1}(r_0^2)$, which implies $\tau_{\lambda'}(R) \preceq_R \tau_{\lambda}(R)$.\\
\quad On the other hand, if we assume $\tau_{\lambda'}(R) \preceq_R \tau_{\lambda}(R)$ then $\tau_{\lambda}^{-1}(r^2) \preceq_r \tau_{\lambda'}^{-1}(r^2)$ is obtained in the same way. Thus we have the assertion since the condition $\tau_{\lambda'}(R) \preceq_R \tau_{\lambda}(R)$ is equivalent to $\varphi_{\lambda'}(R) \preceq_R \varphi_{\lambda}(R)$.
\end{proof}
\begin{lem}
For all $\lambda\in ({\rm Im}\mathbb{H})_0^{\mathbb{Z}}$, we have
\begin{eqnarray}
\lim_{R\to +\infty}\varphi_{\lambda}(R) = +\infty.\nonumber
\end{eqnarray}
\end{lem}
\begin{proof}
From Lemma 3.2, it is enough to show $\lim_{R\to +\infty}\psi_{\lambda}(R) = +\infty$, which follows directly from $\lim_{R\to +\infty}\sharp N_{\lambda}(R) = +\infty$.
\end{proof}
\begin{lem}
For all $\lambda\in ({\rm Im}\mathbb{H})_0^{\mathbb{Z}}$, we have
\begin{eqnarray}
\lim_{R\to +\infty}\frac{\varphi_{\lambda}(R)}{R} = 0.\nonumber
\end{eqnarray}
\end{lem}
\begin{proof}
For each $\varepsilon >0$ there exists a sufficiently large $n_{\varepsilon}\in\mathbb{Z}_{>0}$ such that $\sum_{|n|>n_{\varepsilon}}\frac{1}{|\lambda_n|} < \frac{\varepsilon}{2}$. Hence we have
\begin{eqnarray}
\frac{\varphi_{\lambda}(R)}{R} &=& \sum_{|n|\le n_{\varepsilon}}\frac{1}{R + |\lambda_n|} + \sum_{|n|>n_{\varepsilon}}\frac{1}{R + |\lambda_n|}\nonumber\\
&\le& \sum_{|n|\le n_{\varepsilon}}\frac{1}{R} + \sum_{|n|>n_{\varepsilon}}\frac{1}{|\lambda_n|}\nonumber\\
&\le& \frac{2n_{\varepsilon} + 1}{R} + \frac{\varepsilon}{2}.\nonumber
\end{eqnarray}
Then the inequality $\frac{\varphi_{\lambda}(R)}{R}\le\varepsilon$ holds for any $R\ge\frac{2(2n_{\varepsilon} + 1)}{\varepsilon}$, which implies $\lim_{R\to +\infty}\frac{\varphi_{\lambda}(R)}{R} \le \varepsilon$. The assertion follows by taking $\varepsilon\to 0$.
\end{proof}
\begin{cor}
For all $\lambda\in ({\rm Im}\mathbb{H})_0^{\mathbb{Z}}$, we have
\begin{eqnarray}
\lim_{r\to +\infty}\frac{V_{g_{\lambda}}(r)}{r^4} = 0,\quad \lim_{r\to +\infty}\frac{V_{g_{\lambda}}(r)}{r^3} = +\infty.\nonumber
\end{eqnarray}
\end{cor}
\begin{proof}
It suffices to show that
\begin{eqnarray}
\lim_{r\to +\infty}\frac{\theta_{\lambda}\circ\tau_{\lambda}^{-1}(r^2)}{r^4} = 0,\quad \lim_{r\to +\infty}\frac{\theta_{\lambda}\circ\tau_{\lambda}^{-1}(r^2)}{r^3} = +\infty.\nonumber
\end{eqnarray}
We put $R = \tau_{\lambda}^{-1}(r^2)$ and consider the limit of $R\to +\infty$. Then we have
\begin{eqnarray}
\frac{\theta_{\lambda}\circ\tau_{\lambda}^{-1}(r^2)}{r^4} = \frac{\theta_{\lambda}(R)}{\tau_{\lambda}(R)^2} = \frac{1}{\varphi_{\lambda}(R)},\nonumber
\end{eqnarray}
and
\begin{eqnarray}
\frac{\theta_{\lambda}\circ\tau_{\lambda}^{-1}(r^2)}{r^3} = \frac{\theta_{\lambda}(R)}{\tau_{\lambda}(R)^{\frac{3}{2}}} = \sqrt{\frac{R}{\varphi_{\lambda}(R)}}.\nonumber
\end{eqnarray}
Hence we obtain
\begin{eqnarray}
\lim_{r\to +\infty}\frac{\theta_{\lambda}\circ\tau_{\lambda}^{-1}(r^2)}{r^4} &=& \lim_{R\to +\infty}\frac{1}{\varphi_{\lambda}(R)} = 0,\nonumber\\
\lim_{r\to +\infty}\frac{\theta_{\lambda}\circ\tau_{\lambda}^{-1}(r^2)}{r^3} &=& \lim_{R\to +\infty}\sqrt{\frac{R}{\varphi_{\lambda}(R)}} = +\infty.\nonumber
\end{eqnarray}
from Lemmas 5.7 and 5.8.
\end{proof}

The condition $\varphi_{\lambda}(R)\preceq_R \varphi_{\lambda'}(R)$ is rather difficult to check. But we can describe the sufficient condition for $\varphi_{\lambda}(R)\preceq_R \varphi_{\lambda'}(R)$ easier as follows.
\begin{prop}
Let $\lambda,\lambda'\in ({\rm Im}\mathbb{H})_0^{\mathbb{Z}}$. Suppose that there are some $\alpha >0$ and $R_0>0$ such that $\sharp N_{\lambda}(R) \le \sharp N_{\alpha\lambda'}(R)$ for $R\ge R_0$, where $\alpha\lambda':=(\alpha\lambda'_n)_{n\in\mathbb{Z}}$. Then we have $\varphi_{\lambda}(R)\preceq_R \varphi_{\lambda'}(R)$.
\end{prop}
\begin{proof}
Take bijections $a,b:\mathbb{Z}_{>0}\to\mathbb{Z}$ to be $|\lambda_{a(n)}|\le |\lambda_{a(n+1)}|$, $|\lambda'_{b(n)}|\le |\lambda'_{b(n+1)}|$ for each $n\in\mathbb{Z}_{>0}$. Then the condition $\sharp N_{\lambda}(R) \le \sharp N_{\alpha\lambda'}(R)$ is equivalent to $a^{-1}(N_{\lambda}(R)) \subset b^{-1}(N_{\alpha\lambda'}(R))$.\\
\quad Take $n_0\in\mathbb{Z}_{>0}$ sufficiently large such that $|\lambda_{a(n_0)}|\ge R_0$. Then we have $|\alpha\lambda'_{b(n)}|\le |\lambda_{a(n)}|$ for each $n\ge n_0$ from
\begin{eqnarray}
a^{-1}(N_{\lambda}(|\lambda_{a(n)}|)) \subset b^{-1}(N_{\alpha\lambda'}(|\lambda_{a(n)}|)).\nonumber
\end{eqnarray}
If $n$ is an element of $b^{-1}(N_{\alpha\lambda'}(R))\backslash a^{-1}(N_{\lambda}(R))$ for $R\ge R_0$, then we have
\begin{eqnarray}
\frac{1}{|\lambda_{a(n)}|}\le\frac{1}{R},\quad \frac{1}{|\lambda_{a(n)}|}\le\frac{1}{\alpha|\lambda'_{b(n)}|}.\nonumber
\end{eqnarray}
Thus we have
\begin{eqnarray}
\psi_{\lambda}(R) &=& \sharp N_{\lambda}(R) + \sum_{n\notin a^{-1}(N_{\lambda}(R))}\frac{R}{|\lambda_{a(n)}|}\nonumber\\
&=& \sharp N_{\lambda}(R) + \sum_{n\in b^{-1}(N_{\alpha\lambda'}(R))\backslash a^{-1}(N_{\lambda}(R)^c)}\frac{R}{|\lambda_{a(n)}|} + \sum_{n\notin b^{-1}(N_{\alpha\lambda'}(R))}\frac{R}{|\lambda_{a(n)}|}\nonumber\\
&\le& \sharp N_{\lambda}(R) + (\sharp N_{\alpha\lambda'}(R) - \sharp N_{\lambda}(R)) + \sum_{n\notin b^{-1}(N_{\alpha\lambda'}(R))}\frac{R}{\alpha|\lambda'_{b(n)}|}\nonumber\\
&=& \psi_{\alpha\lambda'}(R)\nonumber
\end{eqnarray}
for $R\ge R_0$. From $\psi_{\lambda'}(R)\simeq_R \psi_{\alpha\lambda'}(R)$, we have $\psi_{\lambda}(R)\preceq_R \psi_{\lambda'}(R)$ which is equivalent to $\varphi_{\lambda}(R)\preceq_R \varphi_{\lambda'}(R)$.
\end{proof}
\section{Examples}
\quad In this section we evaluate $V_{g_{\lambda}}(r)$ concretely for some $\lambda\in ({\rm Im}\mathbb{H})_0^{\mathbb{Z}}$. We fix a bijection $a:\mathbb{Z}_{>0}\to\mathbb{Z}$ and identify $\mathbb{Z}$ with $\mathbb{Z}_{>0}$ throughout this section.\\
\quad Now we fix $\lambda\in ({\rm Im}\mathbb{H})_0^{\mathbb{Z}}$. Assume that $|\lambda_{a(n)}|$ is nondecreasing with respect to $n\in\mathbb{Z}_{>0}$ and there exists a continuous nondecreasing function $\lambda_{\mathbb{R}}:\mathbb{R}_{\ge 0}\to\mathbb{R}_{\ge 0}$ which satisfies $|\lambda_{a(n)}|=\lambda_{\mathbb{R}}(n)$. Then a function $\hat{\varphi}_{\lambda_{\mathbb{R}}}(R):=\int^{\infty}_{0}\frac{Rdx}{R+\lambda_{\mathbb{R}}(x)}$ is strictly increasing and satisfies $\hat{\varphi}_{\lambda_{\mathbb{R}}}(R) \simeq_R \varphi_{\lambda}(R)$. In this case we have $V_{g_{\lambda}}(r)\simeq_r r^2\hat{\tau}_{\lambda_{\mathbb{R}}}^{-1}(r^2)$ where $\hat{\tau}_{\lambda_{\mathbb{R}}}$ is defined by $\hat{\tau}_{\lambda_{\mathbb{R}}}(R):=R\hat{\varphi}_{\lambda_{\mathbb{R}}}(R)$. Now we compute the volume growth of $g_{\lambda}$ in the following two cases.\\
\\
\quad {\bf 1.} Fix $\alpha >1$ and put $\lambda_{\mathbb{R}}(x)=x^{\alpha}$, $\lambda_{a(n)}=\lambda_{\mathbb{R}}(n)i$. Then $\hat{\varphi}_{\lambda_{\mathbb{R}}}$ is given by
\begin{eqnarray}
\hat{\varphi}_{\lambda_{\mathbb{R}}}(R) = \int^{\infty}_{0}\frac{Rdx}{R+x^{\alpha}} = R^{\frac{1}{\alpha}}\int^{\infty}_{0}\frac{dy}{1+y^{\alpha}},\nonumber
\end{eqnarray}
where we put $y=\frac{x}{R^{\frac{1}{\alpha}}}$. Since $\int^{\infty}_{0}\frac{dy}{1+y^{\alpha}}$ is a constant, we have $\hat{\varphi}_{\lambda_{\mathbb{R}}}(R)\simeq_R R^{\frac{1}{\alpha}}$, which gives $\hat{\tau}_{\lambda_{\mathbb{R}}}(R)\simeq_R R^{1 + \frac{1}{\alpha}}$ and $\hat{\tau}_{\lambda_{\mathbb{R}}}^{-1}(r^2)\simeq_r r^{\frac{2\alpha}{\alpha + 1}}$. Hence the volume growth is given by
\begin{eqnarray}
V_{g_{\lambda}}(r)\simeq_r r^{4 - \frac{2}{\alpha + 1}}.\nonumber
\end{eqnarray}
Thus we obtain the following result.
\begin{thm}
There exists a complete $4$-dimensional hyperk\"ahler manifold $(X_{\lambda},g_{\lambda})$ for each $3 <\alpha <4$ whose volume growth is given by
\begin{eqnarray}
V_{g_{\lambda}}(r)\simeq_r r^{\alpha}.\nonumber
\end{eqnarray}
\end{thm}
{\bf 2.} Fix $\alpha >0$ and put $\lambda_{\mathbb{R}}(x)=e^{\alpha x}$, $\lambda_{a(n)}=\lambda_{\mathbb{R}}(n)i$. Then $\hat{\varphi}_{\lambda_{\mathbb{R}}}$ is given by
\begin{eqnarray}
\hat{\varphi}_{\lambda_{\mathbb{R}}}(R) = \int^{\infty}_{0}\frac{Rdx}{R+e^{\alpha x}}.\nonumber
\end{eqnarray}
By putting $y=e^{\alpha x}$, we have
\begin{eqnarray}
\hat{\varphi}_{\lambda_{\mathbb{R}}}(R) &=& \int^{\infty}_{1}\frac{Rdy}{\alpha y(y + R)}\nonumber\\
&=& \frac{1}{\alpha}\log (R+1).\nonumber
\end{eqnarray}
Hence we have $\hat{\varphi}_{\lambda_{\mathbb{R}}}(R) = \frac{1}{\alpha}\log (R+1)$ and $\hat{\tau}_{\lambda_{\mathbb{R}}}(R)=  \frac{1}{\alpha}R\log (R+1)$.
\begin{prop}
Let $\lambda\in ({\rm Im}\mathbb{H})_0^{\mathbb{Z}}$ be as above. Then the volume growth of $g_{\lambda}$ satisfies $V_{g_{\lambda}}(r)\simeq_r \frac{r^4}{\log r}$.
\end{prop}
\begin{proof}
It is enough to see the behavior of $\frac{r^2\hat{\tau}_{\lambda_{\mathbb{R}}}^{-1}(r^2)\log r}{r^4}$ at $r\to +\infty$. Put $R=\hat{\tau}_{\lambda_{\mathbb{R}}}^{-1}(r^2)$. Then we have $r^2= \frac{1}{\alpha}R\log (R+1)$ and $\log r= \frac{1}{2}(\log R + \log \log (R+1) - \log\alpha)$. Thus we have
\begin{eqnarray}
\lim_{r\to +\infty}\frac{r^2\hat{\tau}_{\lambda_{\mathbb{R}}}^{-1}(r^2)\log r}{r^4} &=& \lim_{R\to +\infty}\frac{\alpha}{2}\frac{R(\log R + \log \log (R+1) - \log\alpha)}{R\log (R+1)}\nonumber\\
&=& \frac{\alpha}{2}.\nonumber
\end{eqnarray}
\end{proof}
Thus we have the following theorem.
\begin{thm}
There exists a complete $4$-dimensional hyperk\"ahler manifold $(X_{\lambda},g_{\lambda})$ whose volume growth satisfies
\begin{eqnarray}
\lim_{r\to +\infty}\frac{V_{g_{\lambda}}(r)}{r^4} = 0,\quad \lim_{r\to +\infty}\frac{V_{g_{\lambda}}(r)}{r^{\alpha}} = +\infty\nonumber
\end{eqnarray}
for any $\alpha <4$.
\end{thm}
\section{Taub-NUT deformations}
\quad We consider the volume growth of the Taub-NUT deformations of $(X_{\lambda},g_{\lambda})$ in this section.\\
\quad First of all, we define the Taub-NUT deformations of hyperk\"ahler manifolds with tri-Hamiltonian $S^1$-actions. Let $(X,\omega)$ be a hyperk\"ahler manifold of dimension $4$ with tri-Hamiltonian $S^1$-action, and $\mu:X\to {\rm Im}\mathbb{H}$ be the hyperk\"ahler moment map. An action of $\mathbb{R}$ on $\mathbb{H}$ is defined by $x\mapsto x+\sqrt{s}^{-1} t$ for $x\in\mathbb{H}$ and $t\in\mathbb{R}$ by fixing a constant $s >0$. This action preserves the standard hyperk\"ahler structure on $\mathbb{H}$ and the hyperk\"ahler moment map is given by $\sqrt{s}^{-1}\cdot Im:\mathbb{H}\to {\rm Im}\mathbb{H}$. Then we obtain the hyperk\"ahler quotient with respect to the action of $\mathbb{R}$ on $X\times\mathbb{H}$, that is, the quotient $(\mu^{(s)})^{-1}(\zeta)/\mathbb{R}$ where $\zeta$ is an element of ${\rm Im}\mathbb{H}$ and $\mu^{(s)}:X\times\mathbb{H}\to {\rm Im}\mathbb{H}$ is defined by $\mu^{(s)}(x,y):=\mu(x) + 2\sqrt{s}^{-1} Im(y)$. The hyperk\"ahler structure on $(\mu^{(s)})^{-1}(\zeta)/\mathbb{R}$ is independent of $\zeta$.\\
\quad For each $\zeta\in {\rm Im}\mathbb{H}$ we have an imbedding $\tilde{\iota}_{s,\zeta}:X\to (\mu^{(s)})^{-1}(\zeta)$ defined by $\tilde{\iota}_{s,\zeta}(x):=(x,\frac{\sqrt{s}}{2}(-\mu(x) + \zeta))$ which induces a diffeomorphism $\iota_{s,\zeta}:X\to (\mu^{(s)})^{-1}(\zeta)/\mathbb{R}$. Then we have another hyperk\"ahler structure on $X$ independent of $\zeta$ by the pull-back, which is called Taub-NUT deformation of $\omega$ denoted by $\omega^{(s)}$. If we denote by $g$ the hyperk\"ahler metric of $(X,\omega)$, then we denote by $g^{(s)}$ the hyperk\"ahler metric of $(X,\omega^{(s)})$.\\
\quad There is the pair of a harmonic function and an $S^1$-connection $(\Phi,A)$ corresponding to the hyperk\"ahler structure $\omega$ by Theorem 2.9. Then the corresponding pair to $\omega^{(s)}$ is given by $(\Phi + \frac{s}{4},A)$.\\
\quad The $S^1$-action on $X$ also preserves $\omega^{(s)}$ and $\mu:X\to {\rm Im}\mathbb{H}$ is also the hyperk\"ahler moment map with respect to $\omega^{(s)}$.
\begin{lem}
Let $(X,g)$ be as above and take $p_0,p\in X$. We suppose that $p_0$ is a fixed point by the $S^1$-action. Then we have the inequality
\begin{eqnarray}
d_{g^{(s)}}(p_0,p)^2\ge d_g(p_0,p)^2 + \frac{s}{4}|\mu(p) - \mu(p_0)|^2.\nonumber
\end{eqnarray}
\end{lem}
\begin{proof}
We apply the same argument as Proposition 3.1. Then the assertion follows from
\begin{eqnarray}
d_{g^{(s)}}(p_0,p)^2 &\ge& \inf_{t\in \mathbb{R}}d_{g\times g_{\mathbb{H}}}(\tilde{\iota}_{s,\zeta}(p_0)\cdot t,\tilde{\iota}_{s,\zeta}(p))^2\nonumber\\
&=& \inf_{t\in \mathbb{R}}(d_g(p_0e^{it},p)^2 + \big|\frac{\sqrt{s}}{2}(\mu(p_0) - \mu(p)) + \frac{1}{\sqrt{s}}t\big|^2)\nonumber\\
&=& \inf_{t\in \mathbb{R}}\bigg(d_g(p_0,p)^2 + \frac{s}{4}|\mu(p_0) - \mu(p)|^2 + \frac{t^2}{s}\bigg)\nonumber\\
&=& d_g(p_0,p)^2 + \frac{s}{4}|\mu(p_0) - \mu(p)|^2,\nonumber
\end{eqnarray}
where $g_{\mathbb{H}}$ is the Euclidean metric on $\mathbb{H}$ and $g\times g_{\mathbb{H}}$ is the direct product metric.
\end{proof}
\begin{lem}
Let $(X,g)$ be as above and $B\subset {\rm Im}\mathbb{H}$ be a measurable set. Then we have
\begin{eqnarray}
{\rm vol}_{g^{(s)}}(\mu^{-1}(B)) = {\rm vol}_g(\mu^{-1}(B)) + \frac{\pi s}{4}m_{{\rm Im}\mathbb{H}}(B),\nonumber
\end{eqnarray}
where $m_{{\rm Im}\mathbb{H}}$ is the Lebesgue measure of ${\rm Im}\mathbb{H}$.
\end{lem}
\begin{proof}
It follows directly from Lemma 2.10 and that $\omega^{(s)}$ corresponds to $(\Phi + \frac{s}{4},A)$.
\end{proof}
For $s,C>0$ and $\lambda\in ({\rm Im}\mathbb{H})_0^{\mathbb{Z}}$, put
\begin{eqnarray}
\theta_{\lambda,C}^{(s)}(R):=CR^2\varphi_{\lambda}(R) + \frac{\pi^2s}{3}R^3,\quad \tau_{\lambda,C}^{(s)}(R):=CR\varphi_{\lambda}(R) + \frac{s}{4}R^2.\nonumber
\end{eqnarray}
\begin{prop}
For $\lambda\in ({\rm Im}\mathbb{H})_0^{\mathbb{Z}}$ and $s>0$, we have
\begin{eqnarray}
\limsup_{r\to +\infty}\frac{V_{g_{\lambda}^{(s)}}(p_0,r)}{r^3} \le \frac{8\pi^2}{3\sqrt{s}}.\nonumber
\end{eqnarray}
\end{prop}
\begin{proof}
From Lemma 7.1 and 7.2, we have
\begin{eqnarray}
V_{g_{\lambda}^{(s)}}(p_0,r) \le \theta_{\lambda,P_+}^{(s)}\circ(\tau_{\lambda,Q_-}^{(s)})^{-1}(r^2)\nonumber
\end{eqnarray}
for $r>0$. Then it suffices to show
\begin{eqnarray}
\limsup_{r\to +\infty}\frac{\theta_{\lambda,P_+}^{(s)}\circ(\tau_{\lambda,Q_-}^{(s)})^{-1}(r^2)}{r^3} \le \frac{8\pi^2}{3\sqrt{s}}.\nonumber
\end{eqnarray}
Put $R=(\tau_{\lambda,Q_-}^{(s)})^{-1}(r^2)$. Then we have
\begin{eqnarray}
r^2 =Q_-R\varphi_{\lambda}(R) + \frac{s}{4}R^2 \ge \frac{s}{4}R^2.\nonumber
\end{eqnarray}
Therefore Lemma 5.8 gives
\begin{eqnarray}
\limsup_{r\to +\infty}\frac{\theta_{\lambda,P_+}^{(s)}\circ(\tau_{\lambda,Q_-}^{(s)})^{-1}(r^2)}{r^3} \le \limsup_{R\to +\infty}\frac{8(P_+R^2\varphi_{\lambda}(R) + \frac{\pi^2s}{3}R^3)}{\sqrt{s}^3R^3} = \frac{8\pi^2}{3\sqrt{s}}.\nonumber
\end{eqnarray}
\end{proof}
Next we consider the lower bound for $V_{g_{\lambda}^{(s)}}(r)$. We apply the same way as Section 4. Put $l_{\lambda}^{(s)}:{\rm Im}\mathbb{H}\to\mathbb{R}$ as
\begin{eqnarray}
l_{\lambda}^{(s)}(\zeta):=\int_1^{|\zeta|}\sqrt{\Phi_{\lambda}\bigg(t\frac{\zeta}{|\zeta|}\bigg) + \frac{s}{4}}\ dt\nonumber
\end{eqnarray}
on $|\zeta| >1$, and $l_{\lambda}^{(s)}(\zeta):=0$ on $|\zeta| \le 1$. Then the inequality $d_{g_{\lambda}^{(s)}}(p_0,p)\le L_{\lambda}^{(s)} + l_{\lambda}^{(s)}(\mu_{\lambda}(p))$ holds where $p_0\in X$ is taken as in Sections 3 and 4, and we put $L_{\lambda}^{(s)}:=\sup_{p\in\mu_{\lambda}^{-1}(\bar{B}_1)}d_{g_{\lambda}^{(s)}}(p_0,p)$.
\begin{lem}
Let $C_+ >0$ be as in Proposotion 4.2. Then we have
\begin{eqnarray}
\int_{\Theta\in S^2}l_{\lambda}^{(s)}(R\Theta)dm_{S^2} \le 4\pi(\sqrt{\tau_{\lambda,C_+}(R)} + \sqrt{\frac{sR^2}{4}}).\nonumber
\end{eqnarray}
\end{lem}
\begin{proof}
The assertion follows from
\begin{eqnarray}
\int_{\Theta\in S^2}\bigg(\int_1^R\sqrt{\Phi_{\lambda}(t\Theta) + \frac{s}{4}}\ dt\bigg)dm_{S^2} &\le& \int_{\Theta\in S^2}\bigg(\int_1^R\sqrt{\frac{s}{4}}dt\bigg)dm_{S^2} \nonumber\\
&\ & \quad + \int_{\Theta\in S^2}l_{\lambda}(R\Theta)dm_{S^2}\nonumber\\
&\le& 4\pi\sqrt{\frac{sR^2}{4}} + 4\pi\sqrt{\tau_{\lambda,C_+}(R)}.\nonumber
\end{eqnarray}
\end{proof}
\quad Put $U_{R,T}^{(s)}:=\{\Theta\in S^2; L_{\lambda}^{(s)} + l_{\lambda}^{(s)}(R\Theta)\le\sqrt{\tau_{\lambda,T}(R)} + \sqrt{\frac{sR^2}{4}}\}$.
\begin{lem}
There is a constant $R_0> 0$ such that
\begin{eqnarray}
m_{S^2}(U_{R,T}^{(s)}) \ge \frac{4\pi(\sqrt{T} - \sqrt{2C_+})}{\sqrt{T}}\nonumber
\end{eqnarray}
for any $R\ge R_0$ and $T>2C_+$.
\end{lem}
\begin{proof}
We consider the same argument as in Lemma 4.3. First of all we remark that there exists sufficiently large $R_0>0$ such that
\begin{eqnarray}
\int_{\Theta\in S^2}(L_{\lambda}^{(s)} + l_{\lambda}^{(s)}(R\Theta))dm_{S^2} \le 4\pi(\sqrt{\tau_{\lambda,2C_+}(R)} + \sqrt{\frac{sR^2}{4}})\nonumber
\end{eqnarray}
for any $R\ge R_0$. Then we have
\begin{eqnarray}
\int_{\Theta\in S^2}(L_{\lambda}^{(s)} + l_{\lambda}^{(s)}(R\Theta))dm_{S^2} &=& \int_{\Theta\in U_{R,T}^{(s)}}(L_{\lambda}^{(s)} + l_{\lambda}^{(s)}(R\Theta))dm_{S^2}\nonumber\\
&\ & \  + \int_{\Theta\in S^2\backslash U_{R,T}^{(s)}}(L_{\lambda}^{(s)} + l_{\lambda}^{(s)}(R\Theta))dm_{S^2}\nonumber\\
&\ge& m_{S^2}(U_{R,T}^{(s)})\sqrt{\tau_{\lambda,Q_-}(R) + \frac{sR^2}{4}}\nonumber\\
&\ & \  + (4\pi - m_{S^2}(U_{R,T}^{(s)}))(\sqrt{\tau_{\lambda,T}(R)} + \sqrt{\frac{sR^2}{4}}),\nonumber
\end{eqnarray}
where the constant $Q_->0$ is as in Proposition 3.1. Then an inequality $\sqrt{\tau_{\lambda,Q_-}(R) + \frac{sR^2}{4}}\ge\sqrt{\frac{sR^2}{4}}$ gives
\begin{eqnarray}
\int_{\Theta\in S^2}(L_{\lambda}^{(s)} + l_{\lambda}^{(s)}(R\Theta))dm_{S^2} &\ge& 4\pi\sqrt{\frac{sR^2}{4}}\nonumber\\
&\ &\quad + (4\pi - m_{S^2}(U_{R,T}^{(s)}))\sqrt{\tau_{\lambda,T}(R)}.\nonumber
\end{eqnarray}
Thus we have the conclusion.
\end{proof}
\begin{lem}
For each $R\ge 0$ and $T>0$, $\mu_{\lambda}^{-1}(B_{R,U_{R,T}^{(s)}})$ is a subset of $B_{g_{\lambda}^{(s)}}\bigg(p_0,\sqrt{\tau_{\lambda,T}(R)} + \sqrt{\frac{sR^2}{4}}\bigg)$.
\end{lem}
\begin{proof}
It is shown by the same argument as Lemma 4.4.
\end{proof}
\begin{prop}
For each $\lambda\in ({\rm Im}\mathbb{H})_0^{\mathbb{Z}}$ and $s>0$ we have
\begin{eqnarray}
\liminf_{r\to +\infty}\frac{V_{g_{\lambda}^{(s)}}(p_0,r)}{r^3} \ge \frac{8\pi^2}{3\sqrt{s}}.\nonumber
\end{eqnarray}
\end{prop}
\begin{proof}
For each sufficiently small $\varepsilon >0$ there exists $T_{\varepsilon}>0$ such that
\begin{eqnarray}
m_{S^2}(U_{R,T_{\varepsilon}}^{(s)}) > 4\pi (1 - \varepsilon)\nonumber
\end{eqnarray}
from Lemma 7.5. Since $\frac{\varphi_{\lambda}(R)}{R}$ converges to $0$, there exists $R_{\varepsilon}>0$ such that $T_{\varepsilon}\varphi_{\lambda}(R) < \varepsilon^2 R$ for any $R>R_{\varepsilon}$. Then it holds
\begin{eqnarray}
{\rm vol}_{g_{\lambda}^{(s)}}\bigg(\mu_{\lambda}^{-1}\bigg(B_{R,U_{R,T_{\varepsilon}}^{(s)}}\bigg)\bigg) &\le&  V_{g_{\lambda}^{(s)}}\bigg(p_0,\sqrt{\tau_{\lambda,T_{\varepsilon}}(R)} + \sqrt{\frac{sR^2}{4}}\bigg)\nonumber\\
&\le&  V_{g_{\lambda}^{(s)}}\bigg(p_0,(\varepsilon + \sqrt{\frac{s}{4}}\bigg)R\bigg)
\end{eqnarray}
for $R>R_{\varepsilon}$. On the other hand, we have the following inequality
\begin{eqnarray}
{\rm vol}_{g_{\lambda}^{(s)}}(\mu_{\lambda}^{-1}(B_{R,U})) \ge C_-m_{S^2}(U)R^2\cdot\varphi_{\lambda}(R) + \frac{\pi m_{S^2}(U)}{12}sR^3
\end{eqnarray}
for $U\subset {\rm Im}\mathbb{H}$ from Proposition 4.1. Thus we obtain
\begin{eqnarray}
V_{g_{\lambda}^{(s)}}\bigg(p_0,\bigg(\varepsilon + \sqrt{\frac{s}{4}}\bigg)R\bigg) &\ge& C_-m_{S^2}(U_{R,T_{\varepsilon}}^{(s)})R^2\cdot\varphi_{\lambda}(R) + \frac{\pi m_{S^2}(U_{R,T_{\varepsilon}}^{(s)})}{12}sR^3\nonumber\\
&\ge& \frac{\pi^2}{3}(1 - \varepsilon)sR^3\nonumber
\end{eqnarray}
from $(8)$ and $(9)$. Hence substituting $r=(\varepsilon + \sqrt{\frac{s}{4}})R$ gives
\begin{eqnarray}
\liminf_{r\to +\infty}\frac{V_{g_{\lambda}^{(s)}}(p_0,r)}{r^3} = \liminf_{R\to +\infty}\frac{V_{g_{\lambda}^{(s)}}(p_0,(\varepsilon + \sqrt{\frac{s}{4}})R)}{(\varepsilon + \sqrt{\frac{s}{4}})^3R^3} \ge \frac{1}{3}\frac{\pi^2(1 - \varepsilon)s}{(\varepsilon + \sqrt{\frac{s}{4}})^3}\nonumber
\end{eqnarray}
for any sufficiently small $\varepsilon >0$. Therefore the conclusion follows by taking the limit for $\varepsilon\to 0$.
\end{proof}
From Proposition 7.3 and 7.7, we obtain the followings.
\begin{thm}
Let $\lambda\in ({\rm Im}\mathbb{H})_0^{\mathbb{Z}}$ and $s>0$. Then the volume growth of hyperk\"ahler metric $g_{\lambda}^{(s)}$ is given by
\begin{eqnarray}
\lim_{r\to +\infty}\frac{V_{g_{\lambda}^{(s)}}(r)}{r^3} = \frac{8\pi^2}{3\sqrt{s}}.\nonumber
\end{eqnarray}
\end{thm}

\bibliographystyle{plain}

\end{document}